\documentclass[11pt,a4paper]{article}
\usepackage{amsmath,amssymb,amsfonts,amsthm}

\newcommand{\E}{{\bf{E}}}
\newcommand{\PP}{{\bf{P}}}

\newtheorem{tm}{Theorem}

\newtheorem{lem}{Lemma}

\long\def\symbolfootnote[#1]#2
{\begingroup%
\def\thefootnote{\fnsymbol{footnote}}\footnote[#1]{#2}\endgroup}

\begin{document}

\bibliographystyle{plain}
\parindent=0pt
\centerline{\LARGE \bfseries Assortativity and clustering }
\centerline{\LARGE \bfseries of sparse random intersection graphs}

\par\vskip 3.5em

\centerline{
Mindaugas  Bloznelis\symbolfootnote[1]{Faculty of Mathematics and Informatics, Vilnius University, 03225 Vilnius, Lithuania}
,\  \
Jerzy Jaworski\symbolfootnote[2]{Faculty of Mathematics and Computer Science,
Adam Mickiewicz University, 60-614 Pozna\'n, Poland}
,
\ \
Valentas Kurauskas$^*$
}


\bigskip





\begin{abstract}
We consider sparse random intersection graphs with the property that the clustering coefficient does
not vanish
as the number of nodes tends to infinity.
We find explicit asymptotic expressions for  the correlation coefficient of degrees of adjacent nodes
(called the assortativity
coefficient), the
expected number of common neighbours of adjacent nodes, and  the expected degree of a neighbour of a node
of a given degree $k$.
These expressions are written in terms of the asymptotic degree distribution and, alternatively, in
terms of
the
parameters defining the underlying
random graph model.
\par\end{abstract}

\smallskip
{\bfseries key words:}  assortativity, clustering, power law, random graph, random intersection graph



\section{Introduction}
Assortativity and clustering coefficients are commonly used characteristics describing
statistical dependency of
adjacency relations in real networks (\cite{Newman2002}, \cite{Barrat2000}, \cite{Newman2003}). The
assortativity coefficient of a simple graph  is the Pearson correlation
coefficient between  degrees of the endpoints of a randomly chosen edge. The clustering coefficient
is the conditional probability that three  randomly chosen vertices make up a triangle, given that
the first two are neighbours of the third one.

It is known that many real networks have non-negligible assortativity and clustering coefficients,
and a social network typically has a positive assortativity coefficient (\cite{Newman2002}, \cite{Newman+W+S2002}).
Furthermore, Newman et al. \cite{Newman+W+S2002} remark that the clustering property
(the property that the clustering coefficient attains a non-negligible value) of some
social networks could be explained by the presence of a bipartite graph structure. For example,
in
the actor network two actors are adjacent whenever they have
acted in the same film. Similarly, in the collaboration network   authors are declared
adjacent whenever they have coauthored a paper. These networks exploit the underlying
bipartite graph structure: actors are linked to films, and authors to papers. Such networks are sometimes called
affiliation networks.

In this paper we study assortativity coefficient and its relation to the clustering coefficient in
a~theoretical model of an affiliation network, the so called random intersection graph.
In a~random intersection graph
nodes are prescribed attributes and two nodes are
declared adjacent whenever they share a certain number of attributes (\cite{godehardt2003}, \cite{karonski1999},
see also \cite{Barbour2011}, \cite{Guillaume+L2004}).
An attractive property of random intersection graphs
is that they include power law  degree distributions
and have tunable clustering coefficient
 see \cite{Bloznelis2008}, \cite{Bloznelis2011+}, \cite{Deijfen}, \cite{GJR2010}.
In the present paper we show that the assortativity coefficient of a random intersection graph is
 non-negative. It is positive in the case where the vertex degree distribution has a
finite third moment and the clustering coefficient is positive.
In this case we show  explicit asymptotic expressions for the assortativity coefficient
in terms of moments of the degree distribution as well as in terms of the parameters defining
the random graph. Furthermore, we evaluate the average degree of a neighbour of a vertex of degree $k$, $k=1,2,\dots$, 
(called neighbour connectivity, see \cite{Lee2006}, \cite{Pastor-Satorras2001}), and express it in terms of a related
clustering characteristic, see (\ref{b-h-a-2012++}) below.

Let us rigorously define the network characteristics studied in this paper.
Let ${\cal G}=({\cal V}, {\cal E})$ be a~finite graph on the vertex set
 ${\cal V}$ and with the edge set ${\cal E}$.
The number of neighbours of a~vertex $v$ is denoted $d(v)$.
The number of common neighbours of vertices $v_i$ and $v_j$ is denoted $d(v_i,v_j)$. We are interested
in the correlation between degrees $d(v_i)$ and $d(v_j)$ and the average value of $d(v_i,v_j)$
for adjacent pairs  $v_i\sim v_j$ (here and below '$\sim$' denotes the adjacency relation of ${\cal G}$).
We are also interested in the average values of
$d(v_i)$ and $d(v_i,v_j)$
under the additional condition that the vertex $v_j$ has degree $d(v_j)=k$.

In order to rigorously  define the averaging operation we introduce the
random pair of vertices $(v_1^*, v_2^*)$
drawn uniformly at random from the set of  ordered pairs of distinct vertices.
By $\E f(v_1^*, v_2^*)=\frac{1}{N(N-1)}\sum_{i\not=j}f(v_i,v_j)$ we denote the
 average value of  measurements $f(v_i,v_j)$ evaluated at each ordered pair $(v_i,v_j)$, $i\not=j$.
Here $N=|{\cal V}|$ denotes the total number of vertices.
By $\E^* f(v_1^*, v_2^*)=p_{e*}^{-1}\E \left(f(v_1^*, v_2^*){\mathbb I}_{\{v_1^*\sim v_2^*\}}\right)$
 we denote the  average value over ordered pairs of adjacent vertices.
Here $p_{e*}=\PP(v_1^*\sim v_2^*)$ denotes
the edge probability and ${\mathbb I}_{\{v_i\sim v_j\}}=1$, for $v_i\sim v_j$, and $0$ otherwise.
Furthermore,
$\E^{*k} f(v_1^*, v_2^*)
=
p_{k*}^{-1}\E \left(f(v_1^*, v_2^*){\mathbb I}_{\{v_1^*\sim v_2^*\}}{\mathbb I}_{\{d(v_2^*)=k\}}\right)$,
denotes the average value over ordered pairs of adjacent vertices,
where the second vertex is of degree $k$.
 Here  $p_{k*}=\PP(v_1^*\sim v_2^*, \, d(v_2^*)=k)$.

The average values of $d(v_i)d(v_j)$ and $d(v_i,v_j)$ on  adjacent pairs $v_i\sim v_j$
 are now defined as follows
\begin{displaymath}
g({\cal G})=\E^*d(v_1^*)d(v_2^*),
\qquad
h({\cal G})=\E^*d(v_1^*, v_2^*),
\qquad
h_k({\cal G})=\E^{*k}d(v_1^*, v_2^*).
\end{displaymath}
We also define the average values
\begin{displaymath}
 b({\cal G})=\E^*d(v_1^*),
\qquad
b'({\cal G})=\E^*d^2(v_1^*),
\qquad
b_k({\cal G})=\E^{*k}d(v_1^*)
\end{displaymath}
and the
correlation coefficient
\begin{displaymath}
 r({\cal G})=\frac{g({\cal G})-b^2({\cal G})}{b'({\cal G})-b^2({\cal G})},
\end{displaymath}
 called the assortativity coefficient of ${\cal G}$,
see \cite{Newman2002}, \cite{Newman2003+}.

In the present paper we assume that  our graph  is an instance of a
random graph.
 We consider two random intersection graph models:
active  intersection graph  and passive intersection graph introduced in \cite{godehardt2001}
(we refer to Sections 2 and 3 below
for a detailed description).
Let $G$ denote an instance of a random intersection graph on $N$ vertices. Here and below the number
of vertices is non random. An argument bearing on the law of large numbers
suggests that, for large $N$, we may approximate the  characteristics
 $b(G)$, $b_k(G)$, $h(G)$ and $h_k(G)$ defined for a given instance $G$,
by the corresponding conditional  expectations
\begin{equation}\label{1}
b=\E^{*}d(v_1^*),
\qquad
b_k=\E^{*k} d(v_1^*)
\qquad
h=\E^* d(v_1^*,v_2^*),
\qquad
h_k=\E^{*k} d(v_1^*,v_2^*),
\end{equation}
where now the expected values are taken with respect to the random instance $G$
and the random pair
$(v_1^*, v_2^*)$. We assume that  $(v_1^*, v_2^*)$ is independent of $G$.
Similarly, we may approximate  $r(G)$ by $r=\frac{g-b^2}{b'-b^2}$, where  $b'=\E^{*} d(v_1^*)$
and $g=\E^* d^2(v_1^*)$.

The main  results of this paper are
explicit
asymptotic
expressions as $N\to+\infty$
for the correlation coefficient $r$,
the neighbour connectivity
$b_k$,  and
expected number
of common neighbours $h_k$ defined in (\ref{1}).
As a corollary we obtain that the random intersection graphs have tunable assortativity
coefficient $r\ge 0$.
Another interesting property is expressed by the identity
\begin{equation}\label{b-h-a-2012}
 b_k-h_k=b-h+o(1)
\qquad
{\text{as}}
\qquad
N\to+\infty
\end{equation}
saying that
the average value of the
difference $d(v_i)-d(v_i,v_j)$ of
adjacent vertices $v_i\sim v_j$ is not sensitive to the conditioning on
the neighbour degree $d(v_j)=k$.
That is,
 a neigbour $v_j$ of $v_i$ may affect
the average degree $d(v_i)$ only
by increasing/decreasing
the average number of common neighbours
$d(v_i,v_j)$.
 It is relevant to mention that $h_k=(k-1)\alpha^{[k]}$, where
$\alpha^{[k]}=\PP(v_1^*\sim v_2^*|v_1^*\sim v_3^*, v_2^*\sim v_3^*, d(v_3^*)=k)$ measures the
probability of an edge between two neighbours of a vertex of degree $k$. In particular, we have
\begin{equation}\label{b-h-a-2012++}
 b_k=(k-1)\alpha^{[k]}+b-h+o(1)
\qquad
{\text{as}}
\qquad
N\to+\infty.
\end{equation}

The remaining part of the paper is organized as follows. In Section 2 we introduce the active random graph
and present results for this model.
The passive model is considered in Section 3.
Section 4 contains proofs.

\section{Active intersection graph}
Let $s>0$. Vertices $v_1,\dots, v_n$ of an active intersection graph are represented
by subsets $D_1,\dots, D_n$ of a given ground set
$W=\{w_1,\dots, w_m\}$. Elements of $W$ are called attributes or keys.
Vertices $v_i$ and $v_j$ are declared adjacent if
they share at least $s$ common attributes, i.e.,  we have  $|D_i\cap D_j|\ge s$.

In the {\it active random intersection graph} $G_s(n,m,P)$
every vertex $v_i\in V=\{v_1,\dots, v_n\}$ selects its attribute set $D_i$
independently at random (\cite{godehardt2003}) and all attributes have equal chances to belong to
$D_i$, for each $i=1,\dots, n$. We assume, in addition, that  independent random sets
$D_1,\dots, D_n$ have the same probability distribution. Then, we have
\begin{equation}\label{DP}
 \PP(D_i=A)={\tbinom{m}{|A|}}^{-1}P(|A|),
\end{equation}
for each $A\subset W$, where $P$ is the  common probability distribution of the sizes
$X_i=|D_i|$, $1\le i\le n$ of selected sets. We remark that  $X_i$, $1\le i\le n$ are independent
random variables.

We are interested in the asymptotics of the assortativity coefficient $r$ and  moments (\ref{1})
 in the case where $G_s(n,m,P)$ is sparse and  $n$, $m$ are large.
We address this  question by considering a~sequence of
random graphs $\{G_s(n,m,P)\}_n$, where the integer  $s$ is fixed and  where
$m=m_n$ and $P=P_n$ depend on $n$.
We remark that subsets of $W$ of size $s$ plays a special role,
we call them joints: two vertices are adjacent if their attribute sets share at least one joint.
Our conditions on $P$ are formulated in terms of the number of joints $\tbinom{X_i}{s}$ available to
the typical vertex $v_i$. We denote $a_k=\E{\tbinom{X_1}{s}}^k$. It is convenient to assume that
 as $n\to\infty$  the
rescaled
number of joints $Z_{1}={\tbinom{m}{s}}^{-1/2}n^{1/2}\tbinom{X_1}{s}$
converges in distribution. We
also introduce the $k$-th  moment condition

\smallskip

{\it (i) $Z_{1}$ converges in distribution to  some  random variable $Z$;}

{\it (ii-k)   $0<\E Z^k<\infty$ and \ $\lim_{n\to\infty}\E Z^k_{1}=\E Z^k$.}

We remark that the distribution of $Z$, denoted $P_Z$,
 determines the asymptotic degree distribution
of the sequence $\{G_s(n,m,P)\}_n$
(see \cite{Bloznelis2008}, \cite{Bloznelis2011+}, \cite {Deijfen}, \cite{stark2004}).
We have, under conditions (i), (ii-1) that
\begin{equation}\label{active-degree}
\lim_{m\to\infty} \PP\left( d(v_1)=k \right)=p_k, \qquad
p_k=(k!)^{-1}\E \left( (z_1Z)^ke^{-z_1Z} \right),
\qquad
k=0,1,\dots .
\end{equation}
Here we denote $z_k=\E Z^k$.
Let $d_*$ be a random variable with the probability distribution $\PP(d_*=k)=p_k$, $k=0,1,\dots$.
We call $d_*$ the asymptotic degree.
It follows from (\ref{active-degree}) that the asymptotic degree distribution is a
Poisson mixture, i.e.,  the Poisson
 distribution
with a random (intensity) parameter $z_1Z$.
For example, in the case where  $P_Z$
is degenerate,
i.e., $\PP(Z=z_1)=1$, we obtain the Poisson asymptotic degree distribution.
Furthermore, the asymptotic degree has a  power law when $P_Z$
does.  We denote
\begin{equation}\label{d-delta}
 \delta_i=\E d_*^i,
\qquad
{\overline \delta}_i=\E (d_*)_i,
\qquad
{\text{where}}
\qquad
(x)_i=x(x-1)\cdots(x-i+1).
\end{equation}

Another important characteristic of the sequence $\{G_s(n,m,P)\}_n$ is the asymptotic
ratio $\beta=\lim_{ m \to \infty} {\tbinom{m}{s}}/n$. Together with $P_Z$ it  determines the first order asymptotics
of the clustering coefficient $\alpha=\PP(v_1\sim v_2|v_1\sim v_3, v_2\sim v_3)$,
see \cite{Bloznelis2011+}, \cite{Deijfen}. Under conditions (i), (ii-2), and
\begin{equation}\label{beta}
 {\tbinom{m}{s}}n^{-1}\to\beta\in (0,+\infty)
\end{equation}
 we have
\begin{eqnarray}\label{cc}
\alpha
&
=
&
\frac{a_1}{a_2}+o(1)
\
=
\
\frac{1}{\beta^{1/2}}\frac{\delta_1^{3/2}}{\delta_2-\delta_1}+o(1).
\end{eqnarray}
Furthermore, we have $\alpha=o(1)$ in the case where $\tbinom{m}{s}n^{-1}\to+\infty$.
We remark that
$\alpha=o(1)$ also in the case where the second moment condition (ii-2) fails and  we have
$\E Z^2=+\infty$,  see \cite{Bloznelis2011+}.

To summarize,  the clustering coefficient    $\alpha$ does not vanish as $n,m\to\infty$ whenever the asymptotic degree distribution (equivalently $P_Z$)
has finite second moment and $0<\beta<\infty$.

Our Theorem  \ref{T-Assort-1}, see also Remark 1,  establishes similar properties of the assortativity coefficient~$r$:
it remains bounded away from zero   whenever the asymptotic degree distribution (equivalently $P_Z$)
has finite third moment and $0<\beta<\infty$.

\begin{tm}\label{T-Assort-1} Let $s>0$ be an integer.
Let $m,n\to\infty$. Assume that {\rm (i)} and {\rm (\ref{beta})} are satisfied.
In the case where {\rm (ii-3)}  holds  we have
\begin{eqnarray}\label{T-Assort-1-Formula-1}
 r
&=&
\frac{a_1}{\beta^{-1}(a_1a_3-a_2^2)+a_2}+o(1)
\\
\label{r-b}
&=&
\frac{1}{\sqrt{\beta}}\frac{{\overline\delta}_1^{5/2}}{{\overline\delta}_3{\overline\delta}_1-
{\overline\delta}_2^2+{\overline\delta}_2{\overline\delta}_1}+o(1).
\end{eqnarray}
In the case where {\rm (ii-2)} holds and $\E Z^3=\infty$ we have $r=o(1)$.
\end{tm}

We note that the inequality $a_1a_3\ge a_2^2$, which follows from H\"older's inequality, implies that the ratio
in the right hand side of (\ref{T-Assort-1-Formula-1}) is positive.

{\it Remark 1.}
In the case where (i), (ii-2) hold and $\tbinom{m}{s}n^{-1}\to+\infty$ we have
$r = o(1)$.

Our next result Theorem \ref{T-Assort-2} shows a first order asymptotics
of the neighbour connectivity $b_k$ and the expected number of common neighbours $h_k$.

\begin{tm}\label{T-Assort-2}  Let $s\ge 1$ and $k\ge 0$ be  integers. Let $m,n\to\infty$.
Assume that {\rm (i)}, {\rm (ii-2)} and {\rm (\ref{beta})} hold. We have
\begin{equation}\label{bhactive}
 b=1+\beta^{-1}a_2+o(1),
\qquad
h=\beta^{-1}a_1+o(1)
\end{equation}
and
 \begin{eqnarray}
\label{h(k)active}
h_{k+1}
&=&
\frac{a_1}{\beta}\frac{k}{k+1}\frac{p_{k}}{p_{k+1}}+o(1),
\\
\label{b(k)active}
 b_{k+1}
&=&
1+\beta^{-1}(a_2-a_1)+h_{k+1}+o(1).
\end{eqnarray}
Here $a_1=(\beta\delta_1)^{1/2}+o(1)$ and
$a_2=\beta{\overline\delta}_2/\delta_1+o(1)$.
\end{tm}


We remark that the distribution of the random graph $G_s(n,m,P)$ is invariant under
permutation of its vertices (we refer to this property as the symmetry property in what follows).
Therefore, we have
$b=\E(d(v_1)|v_1\sim v_2)$ and $b_{k+1}=\E(d(v_1)|v_1\sim v_2, d(v_2)=k+1)$. In particular,
the increment $b_{k+1}-b$ shows
how the degree of $v_2$ affects the
average  degree of its neighbour $v_1$. By (\ref{bhactive}), (\ref{b(k)active}), we have
$b_{k+1}-b =\frac{a_1}{\beta}\left(\frac{k}{k+1}\frac{p_k}{p_{k+1}}-1\right)+o(1)$.
In Examples 1 and 2 below we evaluate this quantity for a
power law asymptotic degree distribution and the
Poisson asymptotic degree distribution.

\medskip
{\it Example 1.}  Assume that the asymptotic degree distribution has a power law, i.e., for some
 $c>0$ and $\gamma>3$ we have  $p_k=(c+o(1))k^{-\gamma}$ as $k\to+\infty$. Then
\begin{displaymath}
 \frac{k}{k+1}\frac{p_k}{p_{k+1}}-1=\frac{\gamma-1}{k}+o(k^{-1}).
\end{displaymath}
Hence, for large $k$, we obtain as $n,m\to+\infty$ that
$b_{k+1}-b\approx k^{-1}(\gamma-1)(\delta_1/\beta)^{1/2}$.

\medskip

{\it Example 2.} Assume that the asymptotic degree distribution is Poisson with mean $\lambda>0$, i.e., $p_k=e^{-\lambda}\lambda^k/k!$.
 Then
\begin{displaymath}
 \frac{k}{k+1}\frac{p_k}{p_{k+1}}-1=\frac{k}{\lambda}-1
\end{displaymath}
and, for large $k$, we obtain as $n,m\to+\infty$ that
\begin{equation}\label{poisson-a}
b_{k+1}-b\approx (\lambda\beta)^{-1/2}k.
\end{equation}
Our interpretation of (\ref{poisson-a})  is as follows. We assume, for simplicity,
that  $s=1$. We say that
an~attribute $w\in W$ realises the link
$v_i\sim v_j$,  whenever $w\in D_i\cap D_j$. We note that in a~sparse intersection graph $G_1(n,m,P)$ each link
is realised by
a single attribute with a high probability.
We also remark that in the case of the Poisson asymptotic degree distribution, the sizes of the random sets,
defining intersection graph,
are strongly concentrated about their mean value $a_1$.
Now, by the symmetry property, every element of the attribute set $D_2$  of vertex $v_2$
realises  about
$k/|D_2|\approx k/a_1$  links to some  neighbours  of $v_2$ other than $v_1$.
In particular, the attribute responsible
for the link $v_1\sim v_2$ attracts to $v_1$ some $k/a_1$ neighbours of $v_2$.
Hence, $b_{k+1}-b\approx a_1^{-1}k\approx (\beta \lambda)^{-1/2}k$.

Finally, we remark that (\ref{bhactive}), (\ref{h(k)active}),
and (\ref{b(k)active}) imply (\ref{b-h-a-2012}).

\section{Passive intersection graph}

A  collection $D_1,\dots, D_n$ of subsets of a finite set $W=\{w_1,\dots, w_m\}$
defines  the passive adjacency relation between elements of $W$: $w_i$ and $w_j$
are declared adjacent if $w_i,w_j\in D_k$ for some $D_k$.
In this way we obtain a graph on the vertex set $W$, which we call the passive intersection graph, see \cite{godehardt2003}.
%
%
We assume that  $D_1,D_2,\dots, D_n$ are independent random subsets of $W$ having
the same  probability distribution
(\ref{DP}). In particular, their  sizes $X_i=|D_i|$, $1\le i\le n$ are
independent random variables with the common distribution $P$.
The {\it passive random intersection graph} defined by
the collection $D_1,\dots, D_n$ is denoted $G_1^*(n,m,P)$.

We shall consider a sequence of passive graphs $\{G_1^*(n,m,P)\}_n$, where  $P=P_n$
and $m=m_n$ depend on $n=1,2,\dots$.   We remark that, in the case where $\beta_n=mn^{-1}$ is bounded
and it is
bounded away from zero as $n,m\to+\infty$, the vertex degree distribution  can be
approximated by a compound
Poisson distribution (\cite{Bloznelis2011+}, \cite{JaworskiStark2008}). More precisely,
assuming that $\beta_n\to\beta\in (0,+\infty)$;

\smallskip

{\it (iii) \  $X_{1}$ converges in distribution to a random variable $Z$;}

{\it (iv)
\ $\E Z^{4/3}<\infty$
\
and
\
$\lim_{m\to\infty}\E X^{4/3}_1=\E Z^{4/3}$
}

it is shown in \cite{Bloznelis2011+} that  $d(w_1)$ converges in distribution
 to the compound Poisson random variable
$d_{**}:=\sum_{j=1}^{\Lambda}{\tilde Z}_j$.
Here ${\tilde Z}_1$, ${\tilde Z}_2$,\dots are independent random variables with the distribution
\begin{displaymath}
\PP({\tilde Z}_1=j)=(j+1)\PP(Z=j+1)/\E Z,
\qquad
 j=0,1,\dots,
\end{displaymath}
in the case where $\E Z>0$. In the case where $\E Z=0$ we put $\PP({\tilde Z}_1=0)=1$.
The random variable $\Lambda$ is independent of the sequence  ${\tilde Z}_1$, ${\tilde Z}_2$,\dots
and has Poisson distribution with mean $\E \Lambda=\beta^{-1}\E Z$.

We note that the asymptotic degree $d_{**}$ has a power law whenever $Z$  has a power law.
Furthermore, we have
$\E d^i_{**}<\infty\Leftrightarrow
\E Z^{i+1}<\infty$, $i=1,2,\dots$.

In Theorems \ref{passive1}, \ref{passiveT2}  below
we express
the moments $b$, $h$, $b_k$, $h_k$ and the assortativity coefficient
$r=\frac{g-b^2}{b'-b^2}$ of the random graph $G^*_1(n,m,P)$
in terms of
the
moments
\begin{displaymath}
 y_i=\E(X_1)_i
\qquad
{\text{and}}
\qquad
\delta_{*i}=\E d_{**}^i
\qquad
i=1,2,\dots.
\end{displaymath}

\begin{tm}\label{passive1} Let $n,m\to\infty$.  Assume that {\rm (iii)} holds and

 {\rm (v)}  $\PP(Z\ge 2)>0$,
\
$\E Z^4<\infty$ \ and \   $\lim_{m\to\infty}\E X_1^4=\E Z^4$.

In the case where  $\beta_n\to\beta\in (0,+\infty)$ we have
\begin{eqnarray}\label{rpassive2}
r
&
=
&
\frac{y_2y_4+y_2y_3-y_3^2}
{y_2y_4+y_2y_3-y_3^2+\beta_n^{-1}y_2^2(y_2+y_3)}+o(1)
\\
\label{passive1c+}
&=&
1-\frac{\delta_{*2}\delta_{*1}^2-\delta_{*1}^4}{\delta_{*1}\delta_{*3}-\delta_{*2}^2}
+
o(1).
\end{eqnarray}
%
%

In the case where $\beta_n\to+\infty$ we have $r=1-o(1)$.
In the case where $\beta_n\to0$ and $n\beta_n^3\to+\infty$ we have $r=o(1)$.
\end{tm}

{\it Remark 2}. We
note
that  $y_*:=y_2y_4+y_2y_3-y_3^2$ is always non-negative. Hence, for large $n,m$ we have $r\ge 0$.
To show that $y_*\ge 0$ we combine the identity  $2y_*=\E y(X_1,X_2)$,
 where
\begin{displaymath}
 y(i,j)=y'(i,j)+y'(j,i),
\qquad
y'(i,j)=(i)_2(j)_4+(i)_2(j)_3-(i)_3(j)_3,
\end{displaymath}
with the simple inequality
\begin{displaymath}
  y(i,j)=(i)_2(j)_2\left((i-2)^2+(j-2)^2-2(i-2)(j-2)\right)\ge 0.
\end{displaymath}

{\it Remark 3}. Assuming that $y_2>0$ and $y_2=o(m\beta_n)$ as $m,n\to+\infty$,
Godehardt et al. \cite{GJR2010} showed the following expression for the clustering
coefficient of $G^*_1(n,m,P)$
\begin{equation}\label{GDJKR07-23}
 \alpha=\frac{\beta_n^{-2}m^{-1}y_2^3+y_3}{\beta_n^{-1}y_2^2+y_3}+o(1).
\end{equation}

Now,  assuming that conditions (iii) and (v) hold we compare $\alpha$ and $r$ using
(\ref{rpassive2}) and (\ref{GDJKR07-23}).
For $\beta_n\to\beta\in (0,+\infty)$ we have $r<1$ and
$\alpha=\left(1+y_2^2/(\beta y_3)\right)^{-1}+o(1)<1$.
In the case where $\beta_n\to+\infty$ we have  $r=1-o(1)$ and $\alpha=1-o(1)$. In the case where
$\beta_n\to0$ and $n\beta_n^3\to+\infty$ we have $r=o(1)$ and $\alpha=o(1)$.

Our last result Theorem \ref{passiveT2} shows a first order asymptotics
of the neighbour connectivity $b_k$ and the expected number of common neighbours $h_k$
in the passive random intersection graph.

\begin{tm}\label{passiveT2} Let $m,n\to\infty$. Assume that $\beta_n\to\beta\in (0,+\infty)$ and
{\rm (iii), (v)}  hold.
Then
\begin{eqnarray}
\label{passive2+b}
&&
b
=
1+\beta_n^{-1}y_2+y_2^{-1}y_3+O(n^{-1})=\delta_{*2}\delta_{*1}^{-1}+o(1),
\\
\label{passive2+h}
&&
h
=
y_2^{-1}y_3+O(n^{-1})=\delta_{*2}\delta_{*1}^{-1}-1-\delta_{*1}+o(1).
\end{eqnarray}
Assuming, in addition, that $\PP(d_{**}=k)>0$, where $k>0$ is an integer, we have
\begin{eqnarray}
\label{passive2+h+k}
&&
h_k=k^{-1}\E(d_{2*}|d_{**}=k)+o(1),
\\
\label{passive2+b+k}
&&
b_k=1+\beta^{-1}y_2+h_k+o(1)=1+\delta_{*1}+h_k+o(1).
\end{eqnarray}
Here $d_{2*}=\sum_{1\le i\le \Lambda}({\tilde Z}_i)_2$.
\end{tm}

We remark that (\ref{passive2+b}), (\ref{passive2+h}), (\ref{passive2+h+k}), (\ref{passive2+b+k})
imply (\ref{b-h-a-2012}).

\section{Proofs}
Proofs for active and passive graphs are given in Section 4.1 and Section 4.2 respectively.
We note that the probability distributions of $G_s(n,m,P)$ and $G^*_1(n,m,P)$ are
 invariant under permutations
of the vertex sets. Therefore, for either of these models
we have
\begin{eqnarray}\label{b-h-a+++}
&&
b
\
=\E_{12}d(\omega_1),
\qquad
\qquad
\qquad
\quad
\
h
\
=\E_{12}d(\omega_1,\omega_2),
\\
\nonumber
&&
b_k=\E_{12}(d(\omega_2)|d(\omega_1)=k),
\qquad
h_k= \E_{12}(d(\omega_1,\omega_2)|d(\omega_1)=k).
\end{eqnarray}
Here  $\omega_1\not= \omega_2$ are arbitrary fixed vertices and
$\E_{12}$ denotes the conditional expectation given the event $\omega_1\sim \omega_2$.
In the proof ${\tilde \PP}$ and ${\tilde \E}$ (respectively,  ${\tilde \PP}_*$ and ${\tilde \E}_*$)
denote the conditional probability and expectation given $X_1,\dots, X_n$
(respectively, $D_1, D_2, X_1,\dots, X_n$). Limits are taken as $n$ and $m=m_n$ tend to infinity.
We use the shorthand notation $f_k(\lambda)=e^{-\lambda}\lambda^k/k!$ for the Poisson probability.

\subsection{Active graph}
Before the proof we introduce some more notation.
Then we state and prove  auxiliary lemmas.
Afterwards we prove Theorem \ref{T-Assort-1}, Remark 1 and Theorem \ref{T-Assort-2}.

\medskip

 The conditional expectation given $D_1,D_2$
is denoted
$\E_*$. The conditional expectation given the event $v_1\sim v_2$ is denoted $\E_{12}$.
We denote
\begin{eqnarray}\nonumber
&&
Y_i=\tbinom{X_i}{s},
\qquad
\qquad
\
d_i=d(v_i),
\qquad
\
 d_i'=d_i-1,
\qquad
d_{ij}=d(v_i,v_j),
\\
\label{Ydelta}
&&
{\mathbb I}_i={\mathbb I}_{\{X_i< m^{1/4}\}},
\qquad
{\overline {\mathbb I}}_i=1-{\mathbb I}_i,
\qquad
\delta_{ij}=1-{\overline {\mathbb I}}_i-{\overline {\mathbb I}}_j-(m^{1/2}-1)^{-1}
\end{eqnarray}
and introduce events
\begin{displaymath}
{\cal E}_{ij}'=\{|D_i\cap D_j|=s\},
\qquad
{\cal E}_{ij}''=\{|D_i\cap D_j|\ge s+1\},
\qquad
{\cal E}_{ij}=\{|D_i\cap D_j|\ge s\}.
\end{displaymath}
 Observe that
${\cal E}_{ij}$ is the event that $v_i$ and $v_j$ are adjacent in $G_s(n,m,P)$.
We denote
\begin{displaymath}
p_e=\PP({\cal E}_{ij}),
\qquad
a_i=\E Y_1^i,
\qquad
x_i=\E X_{1}^i,
\qquad
z_i=\E Z^i,
\qquad
{\tilde m}=\tbinom{m}{s},
\qquad
\beta_n=\frac{{\tilde m}}{n}.
\end{displaymath}
We remark that the distributions of
$X_i=X_{ni}$, $Y_i=Y_{ni}$ and $Z_i=Z_{ni}=(n/{\tilde m})^{1/2}Y_{ni}$
depend on~$n$.

\medskip

The following inequality  is referred to as LeCam's lemma, see e.g., \cite{Steele}.
\begin{lem}\label{LeCamLemma} Let $S={\mathbb I}_1+ {\mathbb I}_2+\dots+ {\mathbb I}_n$ be the sum of independent random indicators
with probabilities $\PP({\mathbb I}_i=1)=p_i$. Let $\Lambda$ be Poisson random variable with mean $p_1+\dots+p_n$. The total variation
distance between the distributions $P_S$ and $P_{\Lambda}$ of $S$ and $\Lambda$
\begin{equation}\label{LeCam}
\sup_{A\subset \{0,1,2\dots \}}|\PP(S\in A)-\PP(\Lambda\in A)|\le 2\sum_{i}p_i^2.
\end{equation}
\end{lem}

\begin{lem}\label{xxx} {\rm (\cite{Bloznelis2011+})} Given integers $1\le s\le k_1\le k_2\le m$, let
 $D_1,D_2$ be independent random subsets of the set $W=\{1,\dots, m\}$ such that
$D_1$ (respectively $D_2$) is uniformly distributed in the class of subsets of
$W$ of size $k_1$ (respectively $k_2$).
The  probabilities
$p':=\PP(|D_1\cap D_2|=s)$
and
$p'':=\PP(|D_1\cap D_2|\ge s)$ satisfy
\begin{equation}\label{sp'}
\left(1-\frac{(k_1-s)(k_2-s)}{m+1-k_1}\right)p^*_{k_1,k_2,s}
\
\le
\
p'
\
\le
\
p''
\le
\
 p^*_{k_1,k_2,s},
\end{equation}
Here we denote $p^*_{k_1,k_2,s}={\tbinom{k_1}{s}}{\tbinom{k_2}{s}} {\tbinom{m}{s}}^{-1}$.
\end{lem}

\begin{lem}\label{integrability} Let $s>0$ be an integer. Let $m,n\to\infty$.
Assume that conditions {\rm (i)} and {\rm (ii-3)} hold. Denote
${\tilde X}_{n1}=m^{-1/2}n^{1/(2s)}X_{n1}{\mathbb I}_{\{X_{n1}\ge s\}}$.
We have
\begin{eqnarray}\label{integrability-1}
&&
\lim_{A\to+\infty}\sup_n\E Z_{n1}^3{\mathbb I}_{\{Z_{n1}>A\}}=0,
\\
\label{integrability-2}
&&
\sup_n\E {\tilde X}_{n1}^{3s}<\infty,
\qquad
\lim_{A\to+\infty}\sup_n\E {\tilde X}_{n1}^{3s}{\mathbb I}_{\{{\tilde X}_{n1}>A\}}=0.
\end{eqnarray}
For any $0\le u\le 3$  and any sequence $A_n\to+\infty$ as $n\to\infty$ we have
\begin{equation}\label{integrability-3}
\E Z_{n1}^u{\mathbb I}_{\{Z_{n1}>A_n\}}=o(1),
\qquad
\E {\tilde X}_{n1}^{us}{\mathbb I}_{\{{\tilde X}_{n1}>A_n\}}=o(1).
\end{equation}
\end{lem}

\begin{proof}[Proof of Lemma \ref{integrability}]
 The uniform integrability property (\ref{integrability-1})
of the sequence $\{Z^3_{n1}\}_n$ is a simple consequence of
(i) and (ii-3), see, e.g., Remark 1 in \cite{Bloznelis2008}.
The first and second identity of (\ref{integrability-2}) follows from (ii-3)
and (\ref{integrability-1}) respectively.
Finally, (\ref{integrability-3}) follows from (\ref{integrability-1}) and (\ref{integrability-2}).
\end{proof}

\begin{lem}\label{ABEF} In
$G_s(n,m,P)$ the probabilities of events
${\cal E}_{ij}=\{v_i\sim v_j\}$, ${\cal E}_{12}'$, ${\cal E}_{12}''$, see {\rm (\ref{Ydelta})}, and
${\cal B}_t = \{|D_t\cap (D_1\cup D_2)|\ge s+1\}$
satisfy the inequalities
\begin{eqnarray}
\label{p-e-X}
&&
Y_1Y_2{\tilde m}^{-1}\delta_{12}
\le
{\tilde \PP}({\cal E}'_{12})\le{\tilde \PP}({\cal E}_{12})
\le
Y_1Y_2{\tilde m}^{-1},
\\
\label{pij-lapkritis}
&&
Y_iY_j{\tilde m}^{-1}\delta_{ij}
\le
{\tilde \PP}_*({\cal E}_{ij})={\tilde \PP}({\cal E}_{ij})
\le
Y_iY_j{\tilde m}^{-1},
\qquad
{\text{for}}
\qquad \{i,j\}\not=\{1,2\},
\\
\label{antrakapa++}
&&
{\tilde\PP}({\cal E}''_{12})\le Y_1Y_2X_1X_2({\tilde m}m)^{-1},
\\
\label{Btttt}
&&
{\tilde \PP}_*({\cal B}_t)\le 2^s\left( (s+1)!{\tilde m} m\right)^{-1}Y_tX_t(X_1^{s+1}+X_2^{s+1}).
\end{eqnarray}
We recall that  $Y_i$ and $\delta_{ij}$ are defined in {\rm (\ref{Ydelta})}.
\end{lem}

\begin{proof}[Proof of Lemma \ref{ABEF}]
The right hand side of (\ref{p-e-X}), (\ref{pij-lapkritis}) and inequality (\ref{antrakapa++}) are immediate consequences
of (\ref{sp'}). In order to show the left hand side inequality of (\ref{p-e-X})
 and (\ref{pij-lapkritis}) we apply the left hand side inequality of (\ref{sp'}).
%
We only prove (\ref{p-e-X}). We have, see (\ref{Ydelta}),
\begin{equation}\label{deltaij}
{\tilde \PP}({\cal E}_{12}')
=
{\tilde \E}{\mathbb I}_{{\cal E}_{12}'}
\ge
{\tilde \E}{\mathbb I}_{{\cal E}_{12}'} {\mathbb I}_1{\mathbb I}_2
\ge
{\tilde m}^{-1}Y_1Y_2{\mathbb I}_1{\mathbb I}_2\bigl(1-X_1X_2(m-X_1)^{-1}\bigr)
\ge
{\tilde m}^{-1}Y_1Y_2\delta_{12}.
\end{equation}

In order to show (\ref{Btttt}) we apply the  right-hand side inequality of (\ref{sp'}) and write
\begin{equation}\label{VIIIBt}
{\tilde \PP}_*({\cal B}_t)
\le
\tbinom{|D_1\cup D_2|}{s+1}\tbinom{|D_t|}{s+1}{\tbinom{m}{s+1}}^{-1}
\le
\tbinom{X_1+X_2}{s+1}\tbinom{X_t}{s+1}{\tbinom{m}{s+1}}^{-1}.
\end{equation}
Invoking the inequalities $\tbinom{X_t}{s+1}{\tbinom{m}{s+1}}^{-1}
\le \frac{Y_tX_t}{{\tilde m}m}$ and
\begin{displaymath}
(X_1+X_2)_{s+1}\le (X_1+X_2)^{s+1}\le 2^s(X_1^{s+1}+X_2^{s+1})
\end{displaymath}
 we obtain (\ref{Btttt}).
\end{proof}

\begin{lem}\label{LeCam2} Assume that conditions of Theorem \ref{T-Assort-2} are satisfied.
Let $k\ge 0$ be an integer.
For $d^*_1=\sum_{4\le t\le n}{\mathbb I}_{{\cal E}_{1t}}$ and
$\Delta={\tilde\PP}_*(d_1^*=k)-f_{k}(\beta^{-1}a_1Y_1)$ we have
\begin{equation}\label{d-f+}
\E_*|\Delta|
\le R^*_1+R^*_2+R^*_3+R^*_4,
\end{equation}
where
$R^*_1 = n{\tilde m}^{-1}\E_*Y_1Y_4|1-\delta_{14}|$ and
\begin{displaymath}
R^*_2 =
n^{1/2}{\tilde m}^{-1}a_2^{1/2}Y_1,
\qquad
R^*_3 = a_1Y_1|(n-3){\tilde m}^{-1}-\beta^{-1}|,
\qquad
R^*_4 = 2n{\tilde m}^{-2}a_2Y_1^2.
\end{displaymath}
\end{lem}

\begin{proof}[Proof of Lemma \ref{LeCam2}]
We denote
${\tilde S}={\tilde \E}_*d_1^*=\sum_{4\le t\le n}{\tilde \PP}_*({\cal E}_{1t})$ and
${\tilde S}_1={\tilde m}^{-1}\sum_{4\le t\le n}Y_t$
and write
\begin{displaymath}
\Delta=\Delta_1+\Delta_2,
\qquad
\Delta_1=
{\tilde\PP}_*(d_1^*=k)-f_{k}({\tilde S}),
\qquad
\Delta_2=f_{k}({\tilde S})-f_{k}(\beta^{-1}a_1Y_1).
\end{displaymath}
We have, by Lemma \ref{LeCamLemma},
$|\Delta_1|\le 2\sum_{4\le t\le n}{\tilde \PP}_*^2({\cal E}_{1t})$.
Invoking (\ref{pij-lapkritis}) we obtain $\E_*|\Delta_1|\le R^*_4$.
Next, we apply the mean value theorem
$|f_k(\lambda')-f_k(\lambda'')|\le |\lambda'-\lambda''|$ and write
\begin{equation}
 |\Delta_2|\le |{\tilde S}-\beta^{-1}a_1Y_1|\le r_1^*+r_2^*+R_3^*,
\end{equation}
where $r_1^*=|{\tilde S}-Y_1{\tilde S}_1|$ and $r_2^*=Y_1|{\tilde S}_1-(n-3){\tilde m}^{-1}a_1|$. Note that by (\ref{pij-lapkritis}),
\begin{displaymath}
 r_1^*
\le
\sum_{4\le t\le n}|{\tilde \PP}_*({\cal E}_{1t})-{\tilde m}^{-1}Y_1Y_t|
\le
\sum_{4\le t\le n}{\tilde m}^{-1}Y_1Y_t|1-\delta_{1t}|
\end{displaymath}
and, by symmetry,  $\E_*r_1^*\le R_1^*$. Finally, we have
\begin{displaymath}
 \E_*r_2^*
=
Y_1\E_*|{\tilde S}_1-\E_* {\tilde S_1}|
\le
Y_1\left(\E_*({\tilde S}_1-\E_* {\tilde S_1})^2\right)^{1/2}
\le
R_2^*.
\end{displaymath}
\end{proof}


\begin{lem}\label{momentai2012}  Let $m,n\to\infty$.
Assume {\rm (i)}, {\rm (ii-3)} and {\rm (\ref{beta})} hold. Then
\begin{eqnarray}\label{Y1Y2}
&&
\E_{12}d_1'd_2'=n{\tilde m}^{-1}a_1+n^2{\tilde m}^{-2}a_2^2+o(1),
\\
\label{Y1}
&&
\E_{12}d_1'=n{\tilde m}^{-1}a_2+o(1),
\\
\label{Y1*2}
&&
\E_{12}(d_1')^2=\E_{12}d_1'+ n^2{\tilde m}^{-2}a_1a_3+o(1),
\\
\label{d(v1v2)}
&&
\E_{12}d_{12}=n{\tilde m}^{-1}a_1+o(1).
\end{eqnarray}
\end{lem}

\begin{proof}[Proof of Lemma \ref{momentai2012}] {\it Proof of (\ref{Y1Y2}).}
 In order to prove  (\ref{Y1Y2}) we  write
\begin{equation}\label{kapa1234}
 \E_{12}\, d_1'd_2'=p_{e}^{-1} \E \varkappa,
\qquad
\varkappa: ={\mathbb I}_{{\cal E}_{12}}d_1'd_2',
\qquad
p_e:=\PP({\cal E}_{12})
\end{equation}
and invoke the identities
\begin{eqnarray}\label{r-kappa}
 \E \varkappa
&=&
n{\tilde m}^{-2}a_1^3+n^2{\tilde m}^{-3}a_1^2a_2^2+o({\tilde m}^{-1}),
\\
\label{p-e}
p_e
&=&
{\tilde m}^{-1}
a_1^2(1+o(1)).
\end{eqnarray}
Note that (\ref{p-e})  follows from (\ref{pij-lapkritis}) and (\ref{integrability-3}).
Let us prove (\ref{r-kappa}). To this aim we write
\begin{displaymath}
\E \varkappa
=
\E \left({\mathbb I}_{{\cal E}_{12}}{\tilde \E}_*(d_1'd_2')\right)=\E({\tilde \varkappa}_{1}+{\tilde \varkappa}_{2}),
\end{displaymath}
where
${\tilde \varkappa}_{1}={\mathbb I}_{{\cal E}'_{12}}{\tilde \E}_*d_1'd_2'$ and
${\tilde \varkappa}_{2}={\mathbb I}_{{\cal E}''_{12}}{\tilde \E}_*d_1'd_2'$, and  show that
\begin{equation}\label{vakaras}
 \E{\tilde\varkappa}_1=n{\tilde m}^{-2}a_1^3+n^2{\tilde m}^{-3}a_1^2a_2^2+o({\tilde m}^{-1}),
\qquad
\E{\tilde\varkappa}_2=o({\tilde m}^{-1}).
\end{equation}
Let us prove (\ref{vakaras}).  Assuming that   ${\cal E}_{12}$ holds we can write
$d_i'=\sum_{t=3}^n{\mathbb I}_{{\cal E}_{it}}$, $i=1,2$, and
\begin{equation}\label{At}
{\tilde \E}_*d_1'd_2'=S_1+S_2,
\qquad
S_1=
\sum_{3\le t\le n}{\tilde \PP}_*({\cal E}_{1t}\cap {\cal E}_{2t}),
\qquad
S_2=
2\sum_{3\le t<u\le n} {\tilde \PP}_*({\cal E}_{1t}\cap{\cal E}_{2u}).
\end{equation}
To show the first identity of (\ref{vakaras}) we write $\E{\tilde \varkappa}_1=\E {\mathbb I}_{{\cal E}'_{12}}S_1+\E {\mathbb I}_{{\cal E}'_{12}}S_2=:I_1+I_2$
and evaluate
\begin{equation}\label{valio1}
I_1=
n{\tilde m}^{-2}a_1^3+o( n{\tilde m}^{-2}),
\qquad
I_2=
n^2{\tilde m}^{-3}a_1^2a_2^2+o(n^2{\tilde m}^{-3}).
\end{equation}

We first evaluate $I_1$.
Given $t\ge 3$, consider  events
\begin{equation}\label{ABevents}
{\cal A}_t=\{|(D_1\cap D_2)\cap D_t|=s\}
\qquad
{\text{and}}
\qquad
{\cal B}_t=\{|D_t\cap (D_1\cup D_2)|\ge s+1\}.
\end{equation}
Assuming that  ${\cal E}'_{12}$ holds  we have that
${\cal A}_t$ implies ${\cal E}_{1t}\cap {\cal E}_{2t}$ and ${\cal E}_{1t}\cap {\cal E}_{2t}$ implies
${\cal A}_t\cup {\cal B}_t$.
Hence,
${\tilde \PP}_*({\cal A}_t)
\le
{\tilde \PP}_*({\cal E}_{1t}\cap {\cal E}_{2t})
\le
{\tilde \PP}_*( {\cal A}_t\cup {\cal B}_t)$. Now, we invoke the identity
${\tilde \PP}_*({\cal A}_t)={\tilde m}^{-1}Y_t$ and write
\begin{equation}\label{x!}
{\mathbb I}_{{\cal E}'_{12}}{\tilde m}^{-1}Y_t
=
 {\mathbb I}_{{\cal E}'_{12}}{\tilde \PP}_*({\cal A}_t)
\le
{\mathbb I}_{{\cal E}'_{12}}{\tilde \PP}_*({\cal E}_{1t}\cap {\cal E}_{2t})
\le
 {\mathbb I}_{{\cal E}'_{12}}
\left({\tilde \PP}_*({\cal A}_t)+{\tilde \PP}_*({\cal B}_t)\right).
\end{equation}
>From (\ref{x!}) and (\ref{Btttt})  we obtain, by the symmetry property,
\begin{equation}\label{first-sum}
\frac{n-2}{{\tilde m}}\PP({\cal E}'_{12})\E Y_3
\le
I_1
\le
\frac{n-2}{{\tilde m}}\PP({\cal E}'_{12})\E Y_3+\frac{n-2}{{\tilde m}m}\E {\tilde \PP}({\cal E}'_{12})R_1,
\end{equation}
where $R_1=Y_3X_3 (X_1^{s+1}+X_2^{s+1})$.
 Next, we evaluate ${\tilde \PP}({\cal E}'_{12})$ and
$\PP({\cal E}'_{12})=\E {\tilde \PP}({\cal E}'_{12})$
using  (\ref{p-e-X}):
\begin{displaymath}
 {\tilde m}\PP({\cal E}'_{12})\E Y_3=a_1^3+o(1),
\qquad
{\tilde m}\E {\tilde \PP}({\cal E}'_{12})R_1=O(1).
\end{displaymath}
Combining these relations with (\ref{first-sum}) we obtain the first relation of (\ref{valio1}).

Let us we evaluate $I_2$.
We write
\begin{equation}\label{221}
{\tilde \E}{\mathbb I}_{{\cal E}'_{12}} {\tilde \PP}_*({\cal E}_{1t}\cap {\cal E}_{2u})
=
{\tilde \E}{\mathbb I}_{{\cal E}'_{12}}{\tilde \PP}_*({\cal E}_{1t})\, {\tilde \PP}_*({\cal E}_{2u})
=
{\tilde \PP}({\cal E}'_{12}){\tilde \PP}({\cal E}_{1t})\, {\tilde \PP}({\cal E}_{2u})
\end{equation}
and apply (\ref{p-e-X}) to each probability in the right-hand side. We obtain
\begin{equation}\label{lapkricio15}
{\tilde m}^{-3}(Y_1^2Y_2^2Y_tY_u-R_{tu})
\le
{\tilde \PP}({\cal E}'_{12}){\tilde \PP}({\cal E}_{1t})\, {\tilde \PP}({\cal E}_{2u})
\le
{\tilde m}^{-3}Y_1^2Y_2^2Y_tY_u,
\end{equation}
where  $R_{tu}=Y_1^2Y_2^2Y_tY_u(1-\delta_{12}\delta_{1t}\delta_{2u})$ satisfies $\E R_{tu}=o(1)$,
see (\ref{integrability-3}). Now, by the symmetry property, we obtain from  (\ref{lapkricio15})
the second relation of (\ref{valio1})
\begin{displaymath}
I_2=(n-2)_2\E {\tilde \PP}({\cal E}'_{12}){\tilde \PP}({\cal E}_{1t})\, {\tilde \PP}({\cal E}_{2u})
=n^2{\tilde m}^{-3}a_1^2a_2^2+o(n^2{\tilde m}^{-3}).
\end{displaymath}

To prove the second bound  of (\ref{vakaras})
we write, see
(\ref{At}),  ${\tilde\varkappa}_2={\mathbb I}_{{\cal E}''_{12}}(S_1+S_2)$
and show that
\begin{equation}\label{antrakapa}
I_3:= \E {\mathbb I}_{{\cal E}''_{12}}S_1\le  x_{2s+1}x_{s+1}x_{s}n/({\tilde m}^2m),
\qquad
I_4:=\E {\mathbb I}_{{\cal E}''_{12}}S_2\le  x_{2s+1}^2x_s^2n^2/({\tilde m}^3m).
\end{equation}
Here $x_{2s+1}, x_{s+1}, x_s=O(1)$, by (\ref{integrability-2}).
Let us prove  (\ref{antrakapa}). We have, see (\ref{p-e-X}),
\begin{equation}\label{antrakapa+}
 S_1
\le
\sum_{3\le t\le n}{\tilde \PP}_*({\cal E}_{1t})
\le
\sum_{3\le t\le n}Y_1Y_t{\tilde m}^{-1}.
\end{equation}
Furthermore, by the symmetry property and  (\ref{antrakapa++}), we obtain
\begin{displaymath}
 I_3=\E({\tilde \E}{\mathbb I}_{{\cal E}''_{12}}S_1)
=
\E({\tilde\PP}({\cal E}''_{12})S_1)
\le
(n-2)({\tilde m}^2m)^{-1}\E Y_1^2Y_2Y_3X_1X_2.
\end{displaymath}
Since the expected value in the right hand side does not exceed $ x_{2s+1}x_{s+1}x_{s}$, we obtain
the first bound of (\ref{antrakapa}).
In order to prove the second bound  we write, cf. (\ref{221}),
\begin{displaymath}
{\tilde \E}{\mathbb I}_{{\cal E}''_{12}} {\tilde \PP}_*({\cal E}_{1t}\cap {\cal E}_{2u})
=
{\tilde \PP}({\cal E}''_{12}){\tilde \PP}({\cal E}_{1t})\, {\tilde \PP}({\cal E}_{2u})
\le {\tilde m}^{-3}m^{-1}Y_1^2Y_2^2Y_tY_uX_1X_2.
\end{displaymath}
In the last step we used (\ref{p-e-X}) and (\ref{antrakapa++}). Now, by the symmetry property,
we obtain
\begin{displaymath}
 I_4=\E({\tilde \E}{\mathbb I}_{{\cal E}''_{12}}S_2)
\le (n-2)_2{\tilde m}^{-3}m^{-1}\E Y_1^2Y_2^2Y_3Y_4X_1X_2
\le
n^2{\tilde m}^{-3}m^{-1}x_{2s+1}^2x_s^2.
\end{displaymath}

{\it Proof of  (\ref{Y1}).} We write, by the symmetry property,
\begin{equation}\label{Y1-proof}
 \E_{12}d_1'=p_e^{-1}\E\sum_{3\le t\le n}{\mathbb I}_{{\cal E}_{1t}}{\mathbb I}_{{\cal E}_{12}}
=
(n-2)p_e^{-1}\E {\mathbb I}_{{\cal E}_{13}}{\mathbb I}_{{\cal E}_{12}}
\end{equation}
and evaluate using (\ref{p-e-X}), (\ref{pij-lapkritis})
\begin{displaymath}
\E {\mathbb I}_{{\cal E}_{12}}{\mathbb I}_{{\cal E}_{13}}
=
\E{\tilde \PP}({\cal E}_{12}){\tilde \PP}({\cal E}_{13})
=
{\tilde m}^{-2}\E Y_1^2Y_2Y_3+o({\tilde m}^{-2})
={\tilde m}^{-2}a_1^2a_2+o({\tilde m}^{-2}).
\end{displaymath}
Invoking this relation and (\ref{p-e}) in (\ref{Y1-proof}) we obtain  (\ref{Y1}).

{\it Proof of (\ref{Y1*2})}. Assuming that the event  ${\cal E}_{12}$ holds  we write
\begin{displaymath}
 (d_1')^2=\bigl(\sum_{3\le t\le n}{\mathbb I}_{{\cal E}_{1t}}\bigr)^2
=
d_1'+2\sum_{3\le t<u\le n}{\mathbb I}_{{\cal E}_{1t}}{\mathbb I}_{{\cal E}_{1u}}
\end{displaymath}
and evaluate the expected value
\begin{equation}\label{Y1**2proof}
 \E_{12}(d_1')^2=\E_{12}d_1'
+
p_e^{-1}(n-2)_2\varkappa^*.
\end{equation}
Here
$\varkappa^*=\E {\mathbb I}_{{\cal E}_{12}}{\mathbb I}_{{\cal E}_{13}}{\mathbb I}_{{\cal E}_{14}}$. We have
\begin{equation}\label{2012-07-30-A1}
 \varkappa^*
=
\E{\tilde \PP}({\cal E}_{12})
{\tilde \PP}({\cal E}_{13}){\tilde \PP}({\cal E}_{14})
={\tilde m}^{-3}\E Y_1^3Y_2Y_3Y_4+o({\tilde m}^{-3}).
\end{equation}
In the last step we used (\ref{p-e-X}), (\ref{pij-lapkritis}). Now (\ref{p-e}), (\ref{Y1**2proof}) and (\ref{2012-07-30-A1}) imply
(\ref{Y1*2}).

{\it Proof of (\ref{d(v1v2)})}.
We note that $d_{12}=\sum_{3\le t\le n}{\mathbb I}_{{\cal E}_{1t}}{\mathbb I}_{{\cal E}_{2t}}$ and
$\E {\mathbb I}_{{\cal E}_{12}}d_{12}
=
\E {\mathbb I}_{{\cal E}_{12}}S_1$, see (\ref{At}). Next, we write
\begin{displaymath}
\E_{12}d_{12}=p_e^{-1}
\E {\mathbb I}_{{\cal E}_{12}}S_1=p_e^{-1}(I_1+I_3).
\end{displaymath}
and evaluate the quantity in the right hand side using   (\ref{p-e}) and (\ref{valio1}), (\ref{antrakapa}).

\end{proof}

\begin{proof}[Proof of Theorem \ref{T-Assort-1}]
It is convenient to write $r$ in the form
\begin{equation}\label{eta-xi}
 r=\eta/\xi,
\qquad
{\text{where}}
\qquad
\eta=\E_{12}d_1'd_2'-(\E_{12}d_1')^2,
\qquad
\xi=\E_{12}(d_1')^2-(\E_{12}d_1')^2.
\end{equation}
In the case where (ii-3) holds
we
obtain
(\ref{T-Assort-1-Formula-1}) from  (\ref{Y1Y2}), (\ref{Y1}), (\ref{Y1*2}) and (\ref{eta-xi}).
Then we
derive (\ref{r-b}) from  (\ref{T-Assort-1-Formula-1}) using the identities
\begin{equation}\label{abz}
a_i=\beta^{i/2}z_i+o(1),
\qquad
 {\overline\delta}_i=z_iz_1^i,
\quad
i=1,2,3.
\end{equation}

Now  we consider the case where (ii-2) holds and $\E Z^3=\infty$.
It suffices to show that
\begin{equation}\label{h-g}
 \eta=O(1)
\qquad
{\text{and}}
\qquad
\liminf \xi=+\infty.
\end{equation}
Before the proof of (\ref{h-g}) we remark that (\ref{p-e}) holds under condition (ii-2).
In order to prove the first bound of (\ref{h-g}) we show that $\E_{12}d_1'd_2'=O(1)$
and $\E_{12}d_1'=O(1)$.
To show the first bound we
 write $\E_{12}d_1'd_2'
=p_e^{-1}
\E {\mathbb I}_{{\cal E}_{12}}d_1'd_2'$  and evaluate
\begin{eqnarray}\label{ZZd1d2}
\E {\mathbb I}_{{\cal E}_{12}}d_1'd_2'
&&
=
\E{\mathbb I}_{{\cal E}_{12}}\sum_{3\le t\le n} {\mathbb I}_{{\cal E}_{1t}}{\mathbb I}_{{\cal E}_{2t}}
+
\E{\mathbb I}_{{\cal E}_{12}}\sum_{3\le t,u\le n,
\ t\not=u} {\mathbb I}_{{\cal E}_{1t}}{\mathbb I}_{{\cal E}_{2u}}
\\
\nonumber
&&
=
(n-2)\varkappa_1^*+(n-2)_2\varkappa_2^*,
\end{eqnarray}
where
\begin{eqnarray}\label{ZZd1}
\varkappa_1^*
&=&\E{\mathbb I}_{{\cal E}_{12}}{\mathbb I}_{{\cal E}_{13}}{\mathbb I}_{{\cal E}_{23}}
\le
\E{\mathbb I}_{{\cal E}_{12}}{\mathbb I}_{{\cal E}_{13}}\le {\tilde m}^{-2}a_2a_1^2=O(n^{-2}),
\\
\label{ZZd2}
\varkappa_2^*
&=&\E{\mathbb I}_{{\cal E}_{12}}{\mathbb I}_{{\cal E}_{13}}{\mathbb I}_{{\cal E}_{24}}
\le
{\tilde m}^{-3}a_2^2a_1^2=O(n^{-3}).
\end{eqnarray}
In the last step we used (\ref{p-e-X}) and (\ref{pij-lapkritis}). We note that (\ref{p-e}),
(\ref{ZZd1d2}) and (\ref{ZZd1}), (\ref{ZZd2}) imply $\E_{12}d_1'd_2'=O(1)$.
Similarly, the bound $\E_{12}d_1'=O(1)$ follows from (\ref{p-e}) and the simple bound,
cf. (\ref{Y1-proof}),
\begin{equation}\label{ZZd1+}
 \E_{12}d_1'
=
p_e^{-1}(n-2)\E {\mathbb I}_{{\cal E}_{12}}{\mathbb I}_{{\cal E}_{13}}
\le
p_e^{-1}n{\tilde m}^{-2}a_2a_1^2.
\end{equation}

In order to prove  the second relation of (\ref{h-g}) we show that $\liminf \E_{12}(d_1')^2=+\infty$.
In view of (\ref{p-e}) and (\ref{Y1**2proof}) it suffices to show that
$\liminf n^3\varkappa^*=+\infty$.
It follows from the left-hand side inequality of
 (\ref{sp'}) that
\begin{equation}\label{kapa345}
 n^3\varkappa^*
\ge
n^3 \E
{\mathbb I}_1{\mathbb I}_2{\mathbb I}_3{\mathbb I}_4
{\mathbb I}_{{\cal E}_{12}}{\mathbb I}_{{\cal E}_{13}}{\mathbb I}_{{\cal E}_{14}}
\ge
\E {\mathbb I}_1{\mathbb I}_2{\mathbb I}_3{\mathbb I}_4Z_1^3Z_2Z_3Z_4(1-O(m^{-1/2}))^3,
\end{equation}
where, by the independence of $Z_1,\dots,Z_4$, we have
$
\E {\mathbb I}_1{\mathbb I}_2{\mathbb I}_3{\mathbb I}_4Z_1^3Z_2Z_3Z_4
=
\left(\E
{\mathbb I}_1Z_1^3\right) \left(\E  {\mathbb I}_2Z_2\right)^3$.
Finally, (i) combined with (ii-2) imply $\E  {\mathbb I}_2Z_2=z_1+o(1)$, and (i) combined with
$\E Z^3=\infty$ imply
$\liminf \E {\mathbb I}_1Z_1^3=+\infty$.
\end{proof}


\begin{proof}[Proof of Remark 1] Before the proof we introduce some notation and collect
auxiliary inequalities. We
denote
\begin{displaymath}
 h=h_n=m^{1/2}n^{-1/(4s)},
\qquad
{\tilde h}=
{\tilde h}_n=\tbinom{h}{s}\beta_n^{-1/2}
\end{displaymath}
and observe that, under the assumption of Remark~1,  $\beta_n, h_n, {\tilde h}_n\to+\infty$ and $h_n=o(m^{1/2})$. We further denote
\begin{displaymath}
 {\mathbb I}_{ih}={\mathbb I}_{\{X_i<h\}},
\qquad
{\overline {\mathbb I}}_{ih}=1-{\mathbb I}_{ih},
\qquad
\delta_{ijh}=1-{\overline {\mathbb I}}_{ih}-{\overline {\mathbb I}}_{jh}-\varepsilon_h,
\end{displaymath}
where $\varepsilon_h=h^2(m-h)^{-1}$, and remark that
${\mathbb I}_{ih}={\mathbb I}_{\{Z_i<{\tilde h}\}}$ and $\varepsilon_h=o(1)$.
We observe that conditions (i), (ii-k) imply, for any given
 $u\in(0,k]$, that
\begin{equation}\label{Zu}
\E Z_1^u= z_u+o(1),
\qquad
\E Z_{1}^u{\mathbb I}_{1h}= z_u+o(1),
\qquad
\E Z_{1}^u{\overline {\mathbb I}}_{1h}=o(1).
\end{equation}
Now from (\ref{sp'}) we derive the inequalities
\begin{equation}\label{liepos33}
\E Z_{1}Z_{2}\delta_{12h}
\le
 \E Z_{1}Z_{2}{\mathbb I}_{1h}{\mathbb I}_{2h}(1-\varepsilon_h)
\le n\E{\mathbb I}_{{\cal E}_{12}}{\mathbb I}_{1h}{\mathbb I}_{2h}
\le n\E{\mathbb I}_{{\cal E}_{12}}
\le \E Z_{1}Z_{2}.
\end{equation}
Then invoking in (\ref{liepos33}) relations $\E Z_1=z_1+o(1)$ and
  $\E Z_{1}Z_{2}\delta_{12h}=z_1^2+o(1)$, which follow from  (\ref{Zu}) for $u=1$,
we obtain the relation
\begin{equation}\label{p-e-Z}
 np_e=n\E{\mathbb I}_{{\cal E}_{12}}=z_1^2+o(1).
\end{equation}
Similarly, under conditions (i), (ii-2), we obtain
the relations
\begin{eqnarray}\label{liepos34}
&&
n^2\E{\mathbb I}_{{\cal E}_{12}}{\mathbb I}_{{\cal E}_{13}}
\ \
\
\
 =
z_1^2z_2+o(1),
\\
\label{liepos36}
&&
n^3\E {\mathbb I}_{{\cal E}_{12}}{\mathbb I}_{{\cal E}_{13}}{\mathbb I}_{{\cal E}_{24}}
=z_1^2z_2^2+o(1),
\end{eqnarray}
and, under conditions (i), (ii-3), we obtain
\begin{equation}\label{liepos35}
n^3
\E{\mathbb I}_{{\cal E}_{12}}{\mathbb I}_{{\cal E}_{13}}{\mathbb I}_{{\cal E}_{14}}
=
z_1^3z_3+o(1).
\end{equation}

Let us prove the bound $r=o(1)$ in the case where (i), (ii-2) hold and  $\E Z^3=+\infty$.
In order to prove $r=o(1)$ we show (\ref{h-g}).
Proceeding as in (\ref{ZZd1d2}), (\ref{ZZd1}), (\ref{ZZd2}), (\ref{ZZd1+})
 and using (\ref{p-e-Z}) we show the bounds
$\E_{12}d_1'd_2'=O(1)$ and $\E_{12}d_1'=O(1)$, which imply the first bound of (\ref{h-g}).
Next we show  the second relation of (\ref{h-g}). In view of (\ref{Y1**2proof}) and (\ref{p-e-Z})
it suffices to prove
that $\limsup n^3\varkappa^*=+\infty$. In the proof we proceed similarly as in (\ref{kapa345}) above,
but now we use the product  ${\mathbb I}_{1h}{\mathbb I}_{2h}{\mathbb I}_{3h}{\mathbb I}_{4h}$
instead of
${\mathbb I}_{1}{\mathbb I}_{2}{\mathbb I}_{3}{\mathbb I}_{4}$. We obtain
\begin{displaymath}
 n^3\varkappa^*
\ge
\left(\E {\mathbb I}_{1h}Z_1^3\right)\left(\E {\mathbb I}_{2h}Z_2\right)^3(1-\varepsilon_h)^3.
\end{displaymath}
Here $\E {\mathbb I}_{2h}Z_2=z_1+o(1)$, see (\ref{Zu}). Furthermore,
under conditions (i) and $\E Z^3=+\infty$ we have
$\E {\mathbb I}_{1h}Z_1^3\to+\infty$. Hence, $n^3\varkappa^*\to+\infty$.

\bigskip

Now we prove the bound $r=o(1)$ in the case where (i), (ii-3) hold.
We shall show that
\begin{equation}\label{h-g+}
 \eta=o(1)
\qquad
{\text{and}}
\qquad
\liminf \xi>0.
\end{equation}
Let us prove the second inequality of (\ref{h-g+}).
Combining the first identity of (\ref{ZZd1+}) with  (\ref{p-e-Z}) and (\ref{liepos34}) we obtain
\begin{equation}\label{d1+++}
\E_{12}d_1'=z_2+o(1).
\end{equation}
Next, combining (\ref{Y1**2proof}) with (\ref{p-e-Z}) and (\ref{liepos35}) we obtain
\begin{equation}\label{d1+++2liepos32}
\E_{12}(d_1')^2=\E_{12}d_1'+z_1z_3+o(1).
\end{equation}
It follows from (\ref{d1+++}), (\ref{d1+++2liepos32}) and the inequality
$z_1z_3\ge z_2^2$, which follows from Hoelder's
inequality,
that $\xi=z_2+z_1z_3-z_2^2+o(1)\ge z_2+o(1)$. We have proved the second inequality of (\ref{h-g+}).

Let us prove the first bound of (\ref{h-g+}). In view of (\ref{ZZd1d2}) and (\ref{d1+++}) it suffices
to show that
\begin{equation}\label{liepos38}
p_e^{-1}n^2\varkappa_2^*=z_2^2+o(1),
\qquad
 p_e^{-1}n\varkappa_1^*=o(1).
\end{equation}
We note that the  first relation of (\ref{liepos38})
follows from (\ref{p-e-Z}), (\ref{liepos36}). To prove the second bound of (\ref{liepos38})
we need to show that $\varkappa_1^*=o(n^{-2})$.
We split
\begin{displaymath}
 \varkappa_1^*
=
\E {\mathbb I}_{{\cal E}_{12}'}{\mathbb I}_{{\cal E}_{13}}{\mathbb I}_{{\cal E}_{23}}
+
\E {\mathbb I}_{{\cal E}_{12}''}{\mathbb I}_{{\cal E}_{13}}{\mathbb I}_{{\cal E}_{23}}
\end{displaymath}
and estimate, using (\ref{pij-lapkritis}) and (\ref{antrakapa++}),
\begin{displaymath}
 \E {\mathbb I}_{{\cal E}_{12}''}{\mathbb I}_{{\cal E}_{13}}{\mathbb I}_{{\cal E}_{23}}
\le
\E {\mathbb I}_{{\cal E}_{12}''}{\mathbb I}_{{\cal E}_{13}}
\le
{\tilde m}^{-2}m^{-1}\E Y^2_1Y_2X_1X_2Y_3=O(n^{-2-s^{-1}}).
\end{displaymath}
In the last step we combined the inequality $Y_i^u\le X_i^u{\mathbb I}_{\{X_i\ge s\}}$ and
(\ref{integrability-2}).
Furthermore, using the right-hand side inequality of (\ref{x!}) we write
\begin{displaymath}
 \E {\mathbb I}_{{\cal E}_{12}'}{\mathbb I}_{{\cal E}_{13}}{\mathbb I}_{{\cal E}_{23}}
\le \E {\mathbb I}_{{\cal E}_{12}'}{\tilde m}^{-1}Y_3
+
\E {\mathbb I}_{{\cal E}_{12}'}{\tilde \PP}_*({\cal B}_3)
\end{displaymath}
and estimate, by (\ref{p-e-X}) and (\ref{Btttt}),
\begin{eqnarray}\nonumber
&&
\E {\mathbb I}_{{\cal E}_{12}'}{\tilde m}^{-1}Y_3
\le
{\tilde m}^{-2}\E Y_1Y_2Y_3
=
O(n^{-2}\beta_n^{-1/2}),
\\
\nonumber
&&
\E {\mathbb I}_{{\cal E}_{12}'}{\tilde \PP}_*({\cal B}_3)
\le
{\tilde m}^{-2}m^{-1}\E Y_1Y_2Y_3X_3(X_1^{s+1}+X_2^{s+1})
=
O(n^{-2-s^{-1}}).
\end{eqnarray}
\end{proof}

\medskip

\begin{proof}[Proof of Theorem \ref{T-Assort-2}]
Relations (\ref{bhactive}) follow from (\ref{b-h-a+++}) and (\ref{Y1}), (\ref{d(v1v2)}).

Before the proof of (\ref{h(k)active}) and  (\ref{b(k)active}) we introduce some notation.
Given two sequences of real numbers $\{A_n\}$ and $\{B_n\}$ we write
$A_n\simeq B_n$ (respectively $A_n\simeq 0$) to denote the fact that
$A_n-B_n=o(n^{-2})$ (respectively $A_n=o(n^{-2})$).
We denote $p_*=\PP(v_1\sim v_2, \,d_1'=k)$ and introduce random variables, see (\ref{Ydelta}),
${\mathbb I}^*={\mathbb I}_{1}{\mathbb I}_{2}$,
${\overline{\mathbb I}}^*=1-{\mathbb I}^*$,  and
\begin{displaymath}
\tau_1= {\mathbb I}_{{\cal E}_{12}}\tau,
\qquad
\tau_2={\mathbb I}_{{\cal E}''_{12}}\tau,
\qquad
\tau_3= {\mathbb I}_{{\cal E}'_{12}}{\overline {\mathbb I}}_{{\cal E}_{13}}
{\mathbb I}_{{\cal E}_{23}}{\mathbb I}_{\{d_1^*=k\}},
\qquad
\tau_4={\mathbb I}_{{\cal E}'_{12}}\tau^*,
\qquad
\tau_5= {\mathbb I}_{{\cal E}''_{12}}\tau^*.
\end{displaymath}
Here  $\tau={\mathbb I}_{{\cal E}_{23}}{\mathbb I}_{\{d_1'=k\}}$ and
$\tau^*={\mathbb I}_{{\cal E}_{13}}{\mathbb I}_{{\cal E}_{23}}{\mathbb I}_{\{d_1^*=k-1\}}$,
and
$d^*_1=\sum_{4\le t\le n}{\mathbb I}_{{\cal E}_{1t}}$. We remark that the identity
${\mathbb I}_{{\cal E}_{12}}={\mathbb I}_{{\cal E}'_{12}}+{\mathbb I}_{{\cal E}''_{12}}$
in combination with
 $1={\mathbb I}_{{\cal E}_{13}}+{\overline {\mathbb I}}_{{\cal E}_{13}}$ implies
\begin{equation}\label{tautt}
 \tau_1=\tau_2+\tau_3+\tau_4.
\end{equation}

{\it Proof of (\ref{h(k)active}), (\ref{b(k)active})}. In view of (\ref{b-h-a+++}) we can write
\begin{eqnarray}\label{VIII01h}
h_{k+1}
&=&
\E_{12}(d_{12}|d_1'=k)
=
p_*^{-1}\E{\mathbb I}_{{\cal E}_{12}}{\mathbb I}_{\{d_1'=k\}}d_{12},
\\
\nonumber
b_{k+1}-1
&=&
\E_{12}(d_2'|d_1'=k)=p_*^{-1}\E{\mathbb I}_{{\cal E}_{12}}{\mathbb I}_{\{d_1'=k\}}d_2'.
\end{eqnarray}
Furthermore,  by the symmetry property, we have
\begin{equation}\label{VIII01+h}
\E{\mathbb I}_{{\cal E}_{12}}{\mathbb I}_{\{d_1'=k\}}d_{12}
=
(n-2)\E{\mathbb I}_{{\cal E}_{12}}\tau^*,
\qquad\qquad
\E{\mathbb I}_{{\cal E}_{12}}{\mathbb I}_{\{d_1'=k\}}d_2'
=
(n-2)\E\tau_1.
\end{equation}
We note that (\ref{VIII01h}), (\ref{VIII01+h}) combined with the identities
${\mathbb I}_{{\cal E}_{12}}\tau^*=\tau_4+\tau_5$ and (\ref{tautt}) imply
\begin{equation}\label{VIII01++}
h_{k+1}=(n-2)p_*^{-1}\E(\tau_4+\tau_5),
\qquad
b_{k+1}-1=(n-2)p_*^{-1}\E(\tau_2+\tau_3+\tau_4),
\end{equation}
and observe that (\ref{h(k)active}), (\ref{b(k)active}) follow from (\ref{VIII01++}) and
the relations
\begin{eqnarray}
\label{p*++}
p_*
&=&
n^{-1} (k+1)p_{k+1}+o(n^{-1}),
\\
\label{Btau}
\E \tau_3
&=&
n^{-2}\beta^{-1}(k+1)(a_2-a_1)p_{k+1}+o(n^{-2}),
\\
\label{Atau}
\E \tau_4
&=&
n^{-2}\beta^{-1} k a_1p_k+o(n^{-2}),
\\
\label{0tau}
\E \tau_i
&=&
o(n^{-2}),
\qquad
i=2,5.
\end{eqnarray}
It remains to prove (\ref{p*++}), (\ref{Btau}), (\ref{Atau}),  (\ref{0tau}).

In order to show (\ref{0tau}) we combine the inequalities
 \begin{displaymath}
\tau_i
\le
{\mathbb I}_{{\cal E}_{12}''}{\mathbb I}_{{\cal E}_{23}}
=
{\mathbb I}_{{\cal E}_{12}''}{\mathbb I}_{{\cal E}_{23}}({\mathbb I}^*
+
{\overline {\mathbb I}}^*)
\le
{\mathbb I}_{{\cal E}_{12}''}{\mathbb I}_{{\cal E}_{23}}{\mathbb I}^*
+
{\mathbb I}_{{\cal E}_{12}}{\mathbb I}_{{\cal E}_{23}}{\overline {\mathbb I}}^*
 \end{displaymath}
with the inequalities, which follow from (\ref{pij-lapkritis}) and  (\ref{antrakapa++}),
\begin{eqnarray}\nonumber
&&
 \E{\mathbb I}_{{\cal E}_{12}''}{\mathbb I}_{{\cal E}_{23}}{\mathbb I}^*
\le
\E{\tilde \PP}({\cal E}''_{12}){\tilde \PP}_*({\cal E}_{23}){\mathbb I}^*
\le
({\tilde m}^2m)^{-1}\E Y_1Y_2^2Y_3X_1X_2{\mathbb I}^*=O(n^{-2}m^{-1/2})
\\
\nonumber
&&
\E{\mathbb I}_{{\cal E}_{12}}{\mathbb I}_{{\cal E}_{23}}{\overline {\mathbb I}}^*
\le
\E{\tilde \PP}({\cal E}_{12}){\tilde \PP}_*({\cal E}_{23}){\overline {\mathbb I}}^*
\le
{\tilde m}^{-2}\E Y_1Y_2^2Y_3{\overline {\mathbb I}}^*=o(n^{-2}).
\end{eqnarray}
In the last step we used the bound $\E Y_1Y_2^2Y_3{\overline {\mathbb I}}^*=o(1)$,
which holds under conditions
(i), (ii-2).

{\it Proof of (\ref{Atau}).} We have
\begin{equation}\label{A1tau}
\E \tau_4
=
\E {\mathbb I}_{{\cal E}'_{12}}{\tilde\PP}_*({\cal E}_{23}\cap {\cal E}_{13}){\tilde\PP}_*(d_1^*=k-1).
\end{equation}
We first replace in (\ref{A1tau}) the probability
${\tilde\PP}_*({\cal E}_{23}\cap {\cal E}_{13})$ by ${\tilde \PP}_*({\cal A}_3)=Y_3/{\tilde m}$ using
(\ref{ABevents}), (\ref{x!}).
Then  we replace ${\tilde\PP}_*(d_1^*=k-1)$ by $f_{k-1}(\beta^{-1}a_1Y_1)$ using Lemma \ref{LeCam2}.
Finally, we replace ${\mathbb I}_{{\cal E}'_{12}}$ by ${\tilde m}^{-1}Y_1Y_2$ using (\ref{p-e-X}).
We obtain
\begin{eqnarray}\label{tau4-1}
 \E\tau_4
&\simeq&
{\tilde m}^{-1}
\E{\mathbb I}_{{\cal E}'_{12}}Y_3{\tilde\PP}_*(d_1^*=k-1)
\\
\label{tau4-2}
&\simeq&
{\tilde m}^{-1}
\E{\mathbb I}_{{\cal E}'_{12}}Y_3f_{k-1}(\beta^{-1}a_1Y_1)
\\
\label{tau4-3}
&\simeq&
{\tilde m}^{-2}\E Y_1Y_2Y_3f_{k-1}(\beta^{-1}a_1Y_1)
\\
\label{tau4-4}
&=&
n^{-2}\beta_n^{-2}a_1^2\E Y_1f_{k-1}(\beta^{-1}a_1Y_1).
\end{eqnarray}
Here  (\ref{tau4-1}) follows from the bound
$\E{\mathbb I}_{{\cal E}_{12}'}{\tilde P}_*({\cal B}_3)=o(n^{-2})$. To show this bound we write
\begin{displaymath}
 {\mathbb I}_{{\cal E}_{12}'}{\tilde P}_*({\cal B}_3)=
{\mathbb I}_{{\cal E}_{12}'}{\tilde P}_*({\cal B}_3)
({\mathbb I}^*+{\overline {\mathbb I}}^*)
\le  {\mathbb I}_{{\cal E}_{12}'}{\tilde P}_*({\cal B}_3){\mathbb I}^*
+{\mathbb I}_{{\cal E}_{12}'}{\tilde P}_*({\cal B}_3'){\overline {\mathbb I}}^*,
\end{displaymath}
where ${\cal B}_3'=\{D_3\cap (D_1\cup D_2)|\ge s\}$, and estimate,
see (\ref{p-e-X}), (\ref{Btttt}), (\ref{VIIIBt}),
\begin{eqnarray}\nonumber
 \E {\mathbb I}_{{\cal E}_{12}'}{\tilde P}_*({\cal B}_3){\mathbb I}^*
&\le&
{\tilde m}^{-2}m^{-1} \E Y_1Y_2Y_3X_3(X_1^{s+1}+X_2^{s+1}){\mathbb I}^*
\\
\nonumber
&\le&
{\tilde m}^{-2}m^{-3/4} \E Y_1Y_2Y_3X_3(X_1^{s}+X_2^{s})
\\
\nonumber
&=&
O(n^{-2}m^{-3/4}),
\\
\nonumber
\E {\mathbb I}_{{\cal E}_{12}'}{\tilde P}_*({\cal B}_3'){\overline {\mathbb I}}^*
&\le&
{\tilde m}^{-2} \E Y_1Y_2Y_3(X_1^{s}+X_2^{s}){\overline {\mathbb I}}^*
\\
\nonumber
&\le&
o(n^{-2}).
\end{eqnarray}
Furthermore, (\ref{tau4-2}) follows from the bounds
$\E{\mathbb I}_{{\cal E}'_{12}}Y_3R_j^*=o(n^{-1})$, $1\le j\le 4$, see (\ref{d-f+}).
We show these bound using (\ref{p-e-X}).  For $1\le j\le 3$
the proof is obvious. For $j=4$ we need to show that
$\E{\mathbb I}_{{\cal E}'_{12}}Y_1^2Y_3=o(1)$. For this purpose we write (using the inequality
${\mathbb I}_1Y_1\le {\mathbb I}_1m^{s/4}$)
\begin{displaymath}
{\mathbb I}_{{\cal E}'_{12}}Y_1^2Y_3
=
{\mathbb I}_{{\cal E}'_{12}}Y_1^2Y_3({\mathbb I}_1+{\overline {\mathbb I}}_1)
\le
m^{s/4}{\mathbb I}_{{\cal E}'_{12}}Y_1Y_3{\mathbb I}_1+Y_1^2Y_3{\overline {\mathbb I}}_1
\end{displaymath}
and note that the expected values of both summands in the right hand side tend to zero as
$n\to+\infty$.
Finally, (\ref{tau4-3}) follows from (\ref{p-e-X}) and implies directly (\ref{tau4-4}).

Now we derive (\ref{Atau}) from (\ref{tau4-4}). We observe that
\begin{displaymath}
 k^{-1} \beta^{-1} a_1 \E Y_1f_{k-1}(\beta^{-1}a_1Y_1)
=
\E f_k(\beta^{-1}a_1Y_1)
\to
\E f_k(z_1Z)
\end{displaymath}
(here we use the fact that
the weak convergence of distributions (i) implies the
convergence
of expectations of smooth functions). Furthermore, by (\ref{active-degree}), $\E f_{k}(z_1Z)=p_k$.
 Hence, (\ref{tau4-4}) implies
\begin{equation}\nonumber
 \E \tau_4
\simeq
n^{-2} \beta^{-1} k a_1\E f_{k}(z_1Z)
=
n^{-2}\beta^{-1} k a_1p_k.
\end{equation}

\medskip

{\it Proof of (\ref{Btau}).}
Introduce the event ${\cal C}=\{D_3\cap (D_1\setminus D_2)=\emptyset\}$, probability
${\tilde p}={\tilde \PP}_*({\cal E}_{23}'\cap{\cal C}\cap{\overline{\cal E}}_{13})$,
 and random variable
$H={\tilde m}^{-1}(Y_2-1)Y_3$. We obtain (\ref{Btau}) in several steps. We show that
\begin{eqnarray}
\label{VIII6-1}
\E \tau_3
&\simeq&
\E{\mathbb I}_{{\cal E}_{12}'}{\tilde p}{\mathbb I}_{\{d^*_1=k\}}
\\
\label{VIII6-2}
&\simeq&
\E{\mathbb I}_{{\cal E}_{12}'}H{\mathbb I}_{\{d^*_1=k\}}
\\
\label{VIII6-3}
&\simeq&
\E{\mathbb I}_{{\cal E}_{12}'}Hf_k(\beta^{-1}a_1Y_1)
\\
\label{VIII6-4}
&\simeq&
{\tilde m}^{-1}\E Y_1Y_2Hf_k(\beta^{-1}a_1Y_1)
\\
\label{VIII6-5}
&\simeq&
{\tilde m}^{-2}(a_2-a_1)(k+1)\beta p_{k+1}.
\end{eqnarray}
%
%
%
We note that  (\ref{VIII6-1}) is obtained
by replacing ${\mathbb I}_{{\cal E}_{23}}$ by the product
${\mathbb I}_{{\cal E}'_{23}}{\mathbb I}_{{\cal C}}$ in the formula defining $\tau_3$.
In order to bound the error of this replacement
 we apply the inequality
\begin{equation}\label{23C}
{\mathbb I}_{{\cal E}_{23}'}{\mathbb I}_{{\cal C}}
\le
{\mathbb I}_{{\cal E}_{23}}
\le
{\mathbb I}_{{\cal E}_{23}'}{\mathbb I}_{{\cal C}}
+
{\mathbb I}_{{\cal B}_3}.
\end{equation}
and invoke the bound
$\E{\mathbb I}_{{\cal E}_{12}'}{\overline {\mathbb I}}_{{\cal E}_{13}}{\mathbb I}_{{\cal B}_3}{\mathbb I}_{\{d_1'=k\}}
\le
\E{\mathbb I}_{{\cal E}_{12}'}{\tilde \PP}_*({\cal B}_3)=o(n^{-2})$, see
the proof of (\ref{tau4-1}) above.
We remark that the left hand side inequality of  (\ref{23C}) is obvious. The right hand side
inequality holds because the
event ${\cal E}_{23}$ implies
$({\cal E}_{23}'\cap {\cal C})\cup{\cal B}_3$.

In (\ref{VIII6-2}) we replace ${\tilde p}$ by $H$. To prove (\ref{VIII6-2}) we show that
\begin{equation}\label{VIII-6-7}
 \E{\mathbb I}_{{\cal E}_{12}'}{\tilde p}{\mathbb I}_{\{d^*_1=k\}}
\simeq
\E{\mathbb I}_{{\cal E}_{12}'}{\tilde p}{\mathbb I}_{\{d^*_1=k\}}{\mathbb I}_1
\simeq
\E {\mathbb I}_{{\cal E}_{12}'}H{\mathbb I}_{\{d^*_1=k\}}{\mathbb I}_1
\simeq
\E {\mathbb I}_{{\cal E}_{12}'}H{\mathbb I}_{\{d^*_1=k\}}.
\end{equation}
We remark that the first and third relations follow from the simple bounds,
see (\ref{p-e-X}), (\ref{pij-lapkritis}),
\begin{eqnarray}\nonumber
&&
\E{\mathbb I}_{{\cal E}_{12}'}{\tilde p}{\mathbb I}_{\{d^*_1=k\}}{\overline {\mathbb I}}_1
\le
\E{\mathbb I}_{{\cal E}_{12}'}{\mathbb I}_{{\cal E}_{23}'}{\overline {\mathbb I}}_1
\le
{\tilde m}^{-2}\E Y_1Y_2^2Y_3{\overline {\mathbb I}}_1=o(n^{-2}),
\\
\nonumber
&&
\E {\mathbb I}_{{\cal E}_{12}'}|H|{\mathbb I}_{\{d^*_1=k\}}{\overline {\mathbb I}}_1
\le
{\tilde m}^{-1}\E Y_1Y_2|H|{\overline {\mathbb I}}_1=o(n^{-2}).
 \end{eqnarray}
  In order to show the second relation of  (\ref{VIII-6-7}) we split
\begin{equation}\label{1-2-3}
 {\tilde p}
=
{\tilde \PP}_*({\overline {\cal E}}_{13}|{\cal E}_{23}'\cap {\cal C})
\,
{\tilde \PP}_*({\cal E}_{23}'|{\cal C})
\,
 {\tilde \PP}_*({\cal C})
=:{\tilde p}_1{\tilde p}_2{\tilde p}_3
\end{equation}
and observe that
 ${\tilde p}_1$ is the probability that the random subset
$D_3\cap D_2$ (of size $s$) of $D_2$  does
not match the subset  $D_1\cap D_2$
(we note that $|D_1\cap D_2|=s$, since the event ${\cal E}_{12}'$ holds).
Hence, ${\tilde p}_1=1-Y_2^{-1}$.
Furthermore, from   (\ref{sp'})
 we obtain
\begin{equation}\label{p-3-111}
 {\tilde p}_3=1-{\tilde \PP}_*(D_3\cap(D_1\setminus D_2)\not=\emptyset)\ge
1-{\tilde \PP}_*(D_3\cap D_1\not=\emptyset)
\ge
1- m^{-1}X_1X_3.
\end{equation}

Finally,  ${\tilde p}_2$ is the
probability that the random subset
$D_3$ of $W\setminus (D_1\setminus D_2)$ intersects with  $D_2$  in exactly $s$ elements.
Taking into account that the event ${\cal E}_{12}'$ holds
we obtain (see (\ref{p-e-X}), (\ref{deltaij}))
\begin{equation}\label{p-2-111}
{\tilde m}_1^{-1}Y_2Y_3{\mathbb I}_2{\mathbb I}_3(1-m^{1/2}/(m'-X_1))
\le
{\tilde p}_2
\le
{\tilde m}_1^{-1}Y_2Y_3.
\end{equation}
Here we denote ${\tilde m}_1:=\tbinom{m'}{s}$
and
$m'=|W\setminus (D_1\setminus D_2)|=m-(X_1-s)$. We remark that on the event $\{X_1<m^{1/4}\}$ we have
$m'=m-O(m^{3/4})$. Hence, for large $m$,  (\ref{p-2-111}) implies
\begin{equation}\label{p-2-111+}
{\tilde m}^{-1}Y_2Y_3\delta_{23}{\mathbb I}_1
\le
{\tilde p}_2{\mathbb I}_1
\le
{\tilde m}^{-1}Y_2Y_3{\mathbb I}_1(1+m^{-3/4}(s+o(1)).
\end{equation}
Now, collecting (\ref{p-3-111}), (\ref{p-2-111+}), and the identity ${\tilde p}_1=1-Y_2^{-1}$ in
(\ref{1-2-3}) we obtain the inequalities
\begin{equation}\label{p-tilda-123}
{\mathbb I}_{{\cal E}_{12}'}{\mathbb I}_{1}\delta_{23}
H
(1-m^{-1}X_1X_3)\le
 {\mathbb I}_{{\cal E}_{12}'}{\mathbb I}_{1}{\tilde p}
\le
{\mathbb I}_{{\cal E}_{12}'}{\mathbb I}_{1}
H
(1+O(m^{-3/4}))
\end{equation}
that imply the second relation of  (\ref{VIII-6-7}).

In the proof of (\ref{VIII6-3}), (\ref{VIII6-4}), (\ref{VIII6-5})
 we apply the same argument as in (\ref{tau4-2}),
(\ref{tau4-3}), (\ref{tau4-4}) above.

{\it Proof of (\ref{p*++}).}
We write
\begin{displaymath}
p_*=\E{\mathbb I}_{{\cal E}_{12}} {\tilde \PP}_*(d'_1=k)
=
\E{\tilde \PP}({\cal E}_{12}){\tilde \PP}_*(d'_1=k)
\end{displaymath}
and
in the integrand of the right hand side
we replace
${\tilde \PP}_*(d'_1=k)$ by $f_k(\beta^{-1}a_1Y_1)$
and
${\tilde \PP}({\cal E}_{12})$ by ${\tilde m}^{-1} Y_1Y_2$
 using
(\ref{d-f+}) and
(\ref{p-e-X}), respectively.
\end{proof}



\subsection{Passive graph}
Before the proof we introduce some more notation. Then we present auxiliary lemmas.
Afterwards we prove Theorems \ref{passive1}, \ref{passiveT2}.

By $\E_{ij}$ we denote the conditional expectation given the event
${\cal E}_{ij}=\{w_i\sim w_j\}$.
 Furthermore, we denote
\begin{eqnarray}\nonumber
&&
p_e=\PP({\cal E}_{ij}),
\quad
D_{ij}=D_i\cap D_j,
\quad
X_{ij}=|D_{ij}|,
\quad
x_i=\E X_1^i,
\quad
y_i=\E(X_1)_i,
\quad
u_i=\E(Z)_i.
\end{eqnarray}
For $w\in W$, we denote ${\mathbb I}_{i}(w)={\mathbb I}_{\{w\in D_i\}}$ and
${\overline {\mathbb I}}_{i}(w)=1-{\mathbb I}_{i}(w)$, and
introduce  random variables
\begin{eqnarray}\nonumber
&&
L(w)
=
\sum_{1\le i\le n}l_i(w),
\qquad
l_i(w)={\mathbb I}_{i}(w)(X_i-1),
\\
\nonumber
&&
Q(w)
=
\sum_{1\le i<j\le n}q_{ij}(w),
\qquad
q_{ij}(w)={\mathbb I}_{i}(w){\mathbb I}_{j}(w)(X_{ij}-1),
\\
\nonumber
&&
S_1=\sum_{1\le i\le n}s_i,
\qquad
S_2=\sum_{1\le i<j\le n}s_is_j,
\qquad
s_i={\mathbb I}_{i}(w_1){\mathbb I}_{i}(w_2).
\end{eqnarray}
We say that two vertices $w_i,w_j\in W$ are linked by $D_k$ if $w_i,w_j\in D_k$.
In particular, a set
$D_k$ defines $\tbinom{X_k}{2}$ links between its elements.
We note that $L_t=L(w_t)$ counts the number of links incident to $w_t$. Similarly, $Q_t=Q(w_t)$
counts the number of
different parallel links incident to $w_t$
(a parallel link between $w'$ and $w''$ is realized by a pair of sets
$D_i, D_j$ such that $w', w''\in D_i\cap D_j$). Furthermore,
$S_1$ counts the number of links connecting $w_1$ and $w_2$ and $S_2$
counts the number of different pairs of links connecting $w_1$ and $w_2$.
We denote the degree $d_t=d(w_t)$ and introduce  event ${\cal L}_t=\{L_t=d_t\}$.


\begin{lem}\label{moments3++}
 The factorial moments
${\overline\delta}_{*i}=\E(d_{**})_i$ and $u_i=\E(Z)_i$ satisfy the identities
\begin{equation}\label{solve}
 {\overline\delta}_{*1}=\beta^{-1}u_2,
\qquad
{\overline\delta}_{*2}=\beta^{-2}u_2^2+\beta^{-1}u_3,
\qquad
{\overline\delta}_{*3}=\beta^{-3}u_2^3+3\beta^{-2}u_2u_3+\beta^{-1}u_4.
\end{equation}
\end{lem}

\begin{proof}[Proof of Lemma \ref{moments3++}]
We only show the third identity of (\ref{solve}).
The proof of the first and second
identities is similar, but simpler.
We color $z=z_1+\dots+z_r$ distinct balls using $r$ different colors so that $z_i$
balls receive $i$-th color. The number of triples of balls
\begin{equation}\label{solve1}
 \tbinom{z}{3}
=
\sum_{i\in [r]}\tbinom{z_i}{3}
+
\sum_{i\in [r]}\tbinom{z_i}{2}\sum_{j\in[r]\setminus\{i\}}z_j
+
\sum_{\{i,j,k\}\subset[r]}z_iz_jz_k.
\end{equation}
Here the first sum counts triples of the same color, the second sum counts triples having two different
colors, etc. We apply (\ref{solve1}) to the random variable $\tbinom{d_{**}}{3}$, where
$d_{**}={\tilde Z}_1+\dots+{\tilde Z}_{\Lambda}$. We obtain, by the symmetry property,
\begin{displaymath}
 \E\tbinom{d_{**}}{3}
=
\E \Lambda \E\tbinom{{\tilde Z}_1}{3}
+
\E (\Lambda)_2 \E\tbinom{{\tilde Z}_1}{2}\E{\tilde Z}_1
+
\E\tbinom{\Lambda}{3}(\E{\tilde Z}_1)^3.
\end{displaymath}
Now invoking the simple identities $\E(\Lambda)_i=(\E\Lambda)^i=(u_1\beta^{-1})^i$ and
$\E({\tilde Z}_1)_i=u_{i+1}u_1^{-1}$ we obtain the third identity of (\ref{solve}).
 \end{proof}



\begin{lem}\label{moments1}  We have
\begin{eqnarray}
\label{S1}
&&
\E S_1=n^{-1}\beta_n^{-2}y_2+R'_1,
\\
\label{L1S1}
&&
\E L_1S_1
=
n^{-1}\beta_n^{-2}(y_2+y_3)+n^{-1}\beta_n^{-3}y_2^2+R'_2,
\\
\label{L1L1S1}
&&
\E L_1L_1S_1
=
n^{-1}\beta_n^{-2}(y_2+3y_3+y_4)+3n^{-1}\beta_n^{-3}y_2(y_2+y_3)+n^{-1}\beta_n^{-4}y_2^3
+
R'_3,
\\
\label{L1L2S1}
&&
\E L_1L_2S_1
=
n^{-1}\beta_n^{-2}(y_4+3y_3+y_2)+2n^{-1}\beta_n^{-3}y_2(y_3+y_2)+n^{-1}\beta_n^{-4}y_2^3+R'_4.
\end{eqnarray}
where,  for some absolute constant $c>0$,  we have $|R'_1|\le c n^{-2}\beta_n^{-3}x_2$ and
\begin{eqnarray}
\nonumber
|R'_2|
&
\le
&
cn^{-2}(\beta_n^{-3}+\beta_n^{-4})x_4,
\\
\nonumber
|R'_j|
&
\le
&
cn^{-2}\beta_n^{-3}
(1+\beta_n^{-1}+x_2+\beta_n^{-2}x_2)x_4,
\qquad
j=3, 4.
\end{eqnarray}
\end{lem}

\begin{proof}[Proof of Lemma \ref{moments1}] We only show (\ref{L1L2S1}). The proof of
remaining identities
is similar or simpler.
We write, for $t=1,2$,  $L_t=L(w_t)=l_1(w_t)+L_t'$ and denote
${\overline {\tau}}_j={\tilde \E}s_j=(m)_2^{-1}(X_j)_2$.
 We have, by the symmetry property,
\begin{eqnarray}
\label{L1L2S1+}
\E L_1L_2S_1
&
=
&
n\E s_1L_1L_2,
\\
\nonumber
\E s_1L_1L_2
&
=
&
\E s_1l_1(w_1)l_1(w_2)+2\E s_1l_1(w_1)L_2'+\E s_1L_1'L_2',
\\
\nonumber
\E s_1L_1'L_2'
&
=
&
(n-1)\E s_1l_2(w_1)l_2(w_2)+(n-1)_2\E s_1l_2(w_1)l_3(w_2),
\\
\nonumber
\E s_1l_1(w_1)L_2'
&
=
&
(n-1)\E s_1l_1(w_1)l_2(w_2).
\end{eqnarray}
A straightforward calculation shows that
\begin{eqnarray}\nonumber
{\tilde \E} s_1l_1(w_1)l_1(w_2)
&
=
&
(X_1-1)^2{\overline {\tau}}_1
=
(m)_2^{-1}\left((X_1)_4+3(X_1)_3+(X_1)_2\right),
\\
\nonumber
{\tilde \E} s_1l_1(w_1)l_2(w_2)
&
=
&
m^{-1}(X_1-1)(X_2-1)X_2{\overline {\tau}}_1
=m^{-1}(m)_2^{-1}\left((X_1)_3+(X_1)_2\right)(X_2)_2,
\\
\nonumber
{\tilde \E} s_1l_2(w_1)l_2(w_2)
&
=
&
(X_2-1)^2{\overline {\tau}}_1{\overline {\tau}}_2
=
(m)_2^{-2}(X_1)_2\left((X_2)_4+3(X_2)_3+(X_2)_2\right),
\\
\nonumber
{\tilde \E} s_1l_2(w_1)l_3(w_2)
&
=
&
m^{-2}(X_2)_2(X_3)_2{\overline {\tau}}_1
=
m^{-2}(m)_2^{-1}(X_1)_2(X_2)_2(X_3)_2.
\end{eqnarray}
Invoking  these expressions in the identity $\E s_1l_i(w_t)l_j(w_u)= \E {\tilde \E}s_1l_i(w_t)l_j(w_u)$ we obtain  expressions
for the moments $\E s_1l_i(w_t)l_j(w_u)$. Substituting them in (\ref{L1L2S1+})
we obtain (\ref{L1L2S1}).
\end{proof}




\begin{lem}\label{moments}  We have
\begin{eqnarray}\label{S2S2}
&&
\E S_2
\le
 0.5 n^{-2}\beta_n^{-4}x_2^2,
\\
\label{L1S2}
&&
\E L_1S_2\le n^{-2}\beta_n^{-4}x_2x_3+0.5n^{-2}\beta_n^{-5}x_2^3,
\\
\label{Q1S1}
&&
\E Q_1S_1\le n^{-2}\beta_n^{-4}x_2x_3+0.5n^{-2}\beta_n^{-5}x_2^3,
\\
\label{Q2L1S1}
&&
\E L_1Q_2S_1
=
\E L_2Q_1S_1
\le
n^{-2}\beta_n^{-4}(2x_2x_4+1.5\beta_n^{-1}x_2^2x_3+0.5\beta_n^{-2}x_2^4)+n^{-3}\beta_n^{-6}x_2^2x_4,
\qquad
\\
\label{L1Q1S1}
&&
\E L_1Q_1S_1\le n^{-2}\beta_n^{-4}(x_3^2+x_2x_4)+2.5n^{-2}\beta_n^{-5}x_2^2x_3+0.5n^{-2}\beta_n^{-6}x_2^4,
\\
\label{L1L1S2}
&&
\E L_1L_1S_2\le n^{-2}\beta_n^{-4}(x_3^2+x_2x_4)+2.5n^{-2}\beta_n^{-5}x_2^2x_3+0.5n^{-2}\beta_n^{-6}x_2^4,
\\
\label{L1L2S2}
&&
\E L_1L_2S_2\le n^{-2}\beta_n^{-4}(x_2x_4+x_3^2+2\beta_n^{-1}x_2^2x_3+0.5\beta_n^{-2}x_2^4)+n^{-3}0.5\beta_n^{-6}x_2^2x_4,
\\
\label{Q1R3}
&&
\E Q_1{\mathbb I}_1(w_1)(X_1-1)_2\le 4n^{-2}\beta_n^{-3}y_2(y_3+y_4+\beta^{-1}y_2y_3).
\end{eqnarray}
\end{lem}

\begin{proof}[Proof of Lemma \ref{moments}] We only prove (\ref{Q2L1S1}). The proof of remaining
inequalities is similar or simpler.
In the proof we use the shorthand notation $l_i=l_i(w_1)$ and $q_{ij}=q_{ij}(w_2)$.

To prove (\ref{Q2L1S1}) we write, by the symmetry property,
\begin{eqnarray}\nonumber
\E Q_2L_1S_1
&
=
&
\tbinom{n}{2}\E q_{12}L_1S_1
\\
\nonumber
\E q_{12}L_1S_1
&
=
&
2\E q_{12}l_1S_1+(n-2)\E q_{12}l_3S_1,
\\
\nonumber
\E q_{12}l_1S_1
&
=
&
\E q_{12}l_1s_1+\E q_{12}l_1s_2+(n-2)\E q_{12}l_1s_3,
\\
\nonumber
\E q_{12}l_3S_1
&
=
&
\E q_{12}l_3s_1+\E q_{12}l_3s_2+\E q_{12}l_3s_3+(n-3)\E q_{12}l_3s_4
\end{eqnarray}
and invoke the inequalities
\begin{eqnarray}\nonumber
&&
\E q_{12}l_1s_j\le m^{-4}x_2x_4,
\qquad
\E q_{12}l_3s_j\le m^{-5}x_2^2x_3,
\qquad j=1,2,
\\
\nonumber
&&
\E q_{12}l_1s_3\le m^{-6}x_2^2x_4,
\qquad
\E q_{12}l_3s_3\le m^{-5}x_2^2x_3,
\qquad
\E q_{12}l_3s_4\le m^{-6}x_2^4.
\end{eqnarray}
These inequalities follow from the
 identity $ \E q_{12}l_is_j=\E{\tilde \E} q_{12}l_is_j$ and
the upper bounds for the conditional expectations ${\tilde \E} q_{12}l_is_j$ constructed below.

For $i=1$ and $j=1,2$, we have
\begin{equation}\label{q12l1++}
 {\tilde \E} q_{12}l_1s_j\le {\tilde \E} q_{12}l_1
=
(X_1-1){\tilde \E}q_{12}{\mathbb I}_1(w_1)
\le
m^{-4}X_1^4X_2^2.
\end{equation}
In the first inequality we use  $s_j\le 1$. In the second inequality we use the inequality
\begin{equation}\label{ab++}
{\tilde \E}q_{12}{\mathbb I}_1(w_1)
=
\eta\xi
\le
 m^{-4}X_1^3X_2^2.
\end{equation}
Here $\eta={\tilde \E}\left(X_{12}-1\bigl|{\mathbb I}_1(w_1){\mathbb I}_1(w_2){\mathbb I}_2(w_2)=1\right)$ and
$\xi={\tilde \PP}\left({\mathbb I}_1(w_1){\mathbb I}_1(w_2){\mathbb I}_2(w_2)=1\right)$.
We note that given $X_1,X_2, D_1$, the random variable $\eta$  evaluates the expected number of
elements of $D_1\setminus\{w_2\}$
that belong to the random subset $D_2\setminus\{w_2\}$ (of size $X_2-1$). Hence,
 we have $\eta=(m-1)^{-1}(X_1-1)(X_2-1)$. Furthermore, the probability
\begin{displaymath}
\xi
=
{\tilde \PP}( w_1,w_2\in D_1)
\times
{\tilde \PP}(w_2\in D_2)
=
\frac{\tbinom{X_1}{2}}{\tbinom{m}{2}}
\times
\frac{X_2}{m}.
\end{displaymath}
Combining obtained expressions for $\eta$ and $\xi$ we easily obtain (\ref{ab++}).

For $i=1$ and $j=3$, we write, by the independence of $D_1,D_2$ and $D_3$,
\begin{displaymath}
{\tilde \E}q_{12}l_1s_3
=
({\tilde \E}q_{12}l_1)({\tilde \E} s_3)
\le
m^{-6}X_1^4X_2^2X_3^2.
\end{displaymath}
In the last step we used ${\tilde \E} s_3=(X_3)_2(m)_2^{-1}$ and
${\tilde \E} q_{12}l_1\le m^{-4}X_1^4X_2^2$,
see (\ref{q12l1++}).

For $i=3$ and $j=1,2$, we write
${\tilde\E}q_{12}l_3s_j=({\tilde \E}q_{12}{\mathbb I}_j(w_1))({\tilde\E}l_3)$,
by the independence of $D_1,D_2$ and $D_3$. Invoking the inequalities
\begin{displaymath}
{\tilde\E}l_3=(X_3-1){\tilde \PP}(w_1\in D_3)\le  m^{-1}X_3^2,
\qquad
{\tilde \E}q_{12}{\mathbb I}_1(w_1)=\eta\xi
\le m^{-4}X_1^3X_2^2,
\end{displaymath}
see (\ref{ab++}),
we obtain
${\tilde\E}q_{12}l_3s_1\le m^{-5}X_1^3X_2^2X_3^2$.
Similarly, ${\tilde\E}q_{12}l_3s_2\le m^{-5}X_1^2X_2^3X_3^2$.

For $i,j=3$, we split ${\tilde \E}(q_{12}l_3s_3)=({\tilde\E}q_{12})({\tilde\E}l_3s_3)$ and write
${\tilde\E}q_{12}=\eta_1\xi_1$. Here
\begin{equation}\nonumber
\eta_1=
{\tilde\E}(X_{12}-1|{\mathbb I}_1(w_2){\mathbb I}_2(w_2)=1),
\qquad
\xi_1={\tilde \PP}({\mathbb I}_1(w_2){\mathbb I}_2(w_2)=1).
\end{equation}
Invoking the identities  $\eta_1=(m-1)^{-1}(X_1-1)(X_2-1)$ and $\xi_1=m^{-2}X_1X_2$ we obtain
\begin{equation}\label{q12+++}
 {\tilde\E}q_{12}=\eta_1\xi_1
\le m^{-3}X_1^2X_2^2.
\end{equation}
 Combining (\ref{q12+++}) with the identities
${\tilde\E}l_3s_3=(X_3-1){\tilde\E}s_3=(X_3-1)(X_3)_2(m)_2^{-1}$
we obtain the inequality ${\tilde \E}q_{12}l_3s_3\le m^{-5}X_1^2X_2^2X_3^3$.

For $i=3$ and $j=4$ we write by the independence of $D_1,D_2,D_3,D_4$, and  (\ref{q12+++})
\begin{equation}\nonumber
{\tilde\E}q_{12}l_3s_4
=
\left({\tilde\E}q_{12}\right)\left({\tilde\E}l_3\right)\left({\tilde\E}s_4\right)
\le
\left(m^{-3}X_1^2X_2^2\right)
\left(m^{-1}X_3^2\right)
\left(m^{-2}X_4^2\right)
=m^{-6}X_1^2X_2^2X_3^2X_4^2.
\end{equation}
\end{proof}


\begin{proof}[Proof of Theorem \ref{passive1}]  In order to show (\ref{rpassive2}) we write
\begin{equation}\label{rpassive}
 r=\frac{\E_{12}d_1d_2-(\E_{12}d_1)^2}{\E_{12}d_1^2-(\E_{12}d_1)^2}
=
\frac{p_e\E d_1d_2{\mathbb I}_{{\cal E}_{12}}-(\E d_1{\mathbb I}_{{\cal E}_{12}})^2}
{p_e\E d_1^2 {\mathbb I}_{{\cal E}_{12}}-(\E d_1{\mathbb I}_{{\cal E}_{12}})^2}
\end{equation}
and invoke the expressions
\begin{eqnarray}\label{LSQ}
p_e
&
=
&
\E S_1+O(n^{-2}\beta_n^{-4}),
\\
\nonumber
\E d_1{\mathbb I}_{{\cal E}_{12}}
&
=
&
\E L_1S_1+O(n^{-2}\beta_n^{-4}(1+\beta_n^{-1})),
\\
\nonumber
\E d_1^2{\mathbb I}_{{\cal E}_{12}}
&
=
&
\E L_1^2S_1+O(n^{-2}\beta_n^{-4}(1+\beta_n^{-2})),
\\
\nonumber
\E d_1d_2{\mathbb I}_{{\cal E}_{12}}
&
=
&
\E L_1L_2S_1+O(n^{-2}\beta_n^{-4}(1+\beta_n^{-2})).
\end{eqnarray}
Now the identities of Lemma \ref{moments1} complete the proof of (\ref{rpassive2}).

Let us prove (\ref{LSQ}).
We first write, by the inclusion-exclusion,
\begin{eqnarray}\label{S}
 &&
S_1-S_2\le {\mathbb I}_{{\cal E}_{12}}\le S_1,
\\
\label{dt}
&&
 L_t-Q_t\le d_t\le L_t.
\end{eqnarray}
Then  we derive from (\ref{dt}) the inequalities
\begin{equation}\label{dtt}
0\le L_1L_2-d_1d_2\le L_1Q_2+L_2Q_1
\qquad
{\text{and}}
\qquad
0\le L_1^2-d_1^2\le 2L_1Q_1,
\end{equation}
which, in combination with  (\ref{S}) and (\ref{dt}), imply the inequalities
\begin{eqnarray}
\label{d1}
&&
0\le L_1S_1-d_1{\mathbb I}_{{\cal E}_{12}}\le L_1S_2+Q_1S_1,
\\
\nonumber
&&
0\le L_1^2S_1-d_1^2{\mathbb I}_{{\cal E}_{12}}\le L_1^2S_2+2L_1Q_1S_1,
\\
\nonumber
&&
0\le L_1L_2S_1-d_1d_2{\mathbb I}_{{\cal E}_{12}} \le L_1L_2S_2+L_1Q_2S_1+L_2Q_1S_1.
\end{eqnarray}
Finally, invoking the upper bounds for the expected values of the
 quantities in the right hand sides of (\ref{d1}) shown in
 Lemma \ref{moments}, we obtain (\ref{LSQ}).

Now we derive (\ref{passive1c+}) from (\ref{rpassive2}). Firstly,
using the fact that
(iii), (v) imply the convergence of moments  $\E(X_1)_i\to \E (Z)_i$, for $i=2,3,4$,
we replace the moments $y_i$ by $u_i=\E(Z)_i$ in (\ref{rpassive2}). Secondly, we replace $u_i$ by
their expressions via $\delta_{*i}$. For this purpose we
 solve for $u_2$, $u_3$, $u_4$
from (\ref{solve})
and invoke the identities
\begin{equation}\label{solve+++}
 {\overline\delta}_{*1}=\delta_{*1},
\qquad
{\overline\delta}_{*2}=\delta_{*2}-\delta_{*1},
\qquad
{\overline\delta}_{*3}=\delta_{*3}-3\delta_{*2}+2\delta_{*1}.
\end{equation}

For $\beta_n\to+\infty$  relation (\ref{rpassive2}) remains valid and it implies
$r=1+o(1)$.

For  $\beta_n\to 0$   the
 condition $n\beta^3_n\to+\infty$ on the rate of decay
of $\beta_n$ ensures that  the remainder terms of  (\ref{LSQ}) and
Lemma \ref{moments1} are negligibly small. In particular, we derive (\ref{rpassive2}) using
 the same argument as above.
Letting $\beta_n\to 0$ in (\ref{rpassive2}) we obtain the bound $r=o(1)$.
\end{proof}



\begin{proof}[Proof of Theorem \ref{passiveT2}] Before the proof we introduce some notation.
We denote
\begin{displaymath}
 H=\sum_{1\le i\le n}{\mathbb I}_i(w_1)(X_i-1)_2,
\qquad
p_{ke}=\PP(w_2\sim w_1,d_1=k).
\end{displaymath}
Given $w_i, w_j\in W$ we write $d_{ij}=d(w_i,w_j)$.
A common neighbour $w$ of $w_i$ and $w_j$ is called black
 if  $\{w,w_i,w_j\}\subset D_r$ for some $1\le r\le n$, otherwise it is called red.
Let $d'_{ij}$ and $d''_{ij}$ denote the numbers of black and red common neighbours, so that
 $d'_{ij}+d''_{ij}=d_{ij}$.
Let $w_*$ be a vertex drawn uniformly at random from the set $W'=W\setminus  \{w_1\}$.
By  $d'_{1*}$ we denote the  number of black common neighbours of $w_1$ and $w_*$.
By ${\cal E}_{1*}$ we denote the event $\{w_1\sim w_*\}$. We assume that $w_*$ is independent
of the collection
of random sets $D_1\dots, D_n$ defining the adjacency relation of our graph.

In the proof we use the identity,  which follows from (\ref{S1}), (\ref{LSQ}),
\begin{equation}\label{passiveL1+}
p_e=
n^{-1} \beta_n^{-2}y_2+O(n^{-2}).
\end{equation}
  We also use the identities, which follow from (\ref{solve}) and (\ref{solve+++})
\begin{equation}\label{atsibodo}
 1+\beta^{-1}u_2+u_2^{-1}u_3=\delta_{*2}\delta_{*1}^{-1},
\qquad
\beta^{-1}u_2=\delta_{*1}.
\end{equation}

We remark that (\ref{atsibodo})  in combination with relations $y_i\to u_i$ as $n,m\to+\infty$,
imply the right hand side
relations of (\ref{passive2+b}), (\ref{passive2+h}) and (\ref{passive2+b+k}).

Now we prove the left hand side relations of (\ref{passive2+b}), (\ref{passive2+h})
and (\ref{passive2+b+k}),
and the relation  (\ref{passive2+h+k}).


In order to show (\ref{passive2+b}) we write
$b=p_e^{-1}\E d_1{\mathbb I}_{{\cal E}_{12}}$  and
invoke identities  (\ref{LSQ}), (\ref{L1S1}) and (\ref{passiveL1+}).


{\it Proof of (\ref{passive2+h})}. We write $h=p_e^{-1}\E d_{12}{\mathbb I}_{{\cal E}_{12}}$ and
evaluate
\begin{equation}\label{VIII13}
 \E d_{12}{\mathbb I}_{{\cal E}_{12}}=n^{-1}\beta_n^{-2}y_3+O(n^{-2}).
\end{equation}
Combining (\ref{passiveL1+}) with (\ref{VIII13}) we obtain (\ref{passive2+h}).
Let us show (\ref{VIII13}).
Using the identity
\begin{equation}\label{identityVIII}
 d_{12}
=d_{12}'+d_{12}''
=d_{12}'{\mathbb I}_{{\cal L}_1}
+
d_{12}'{\overline {\mathbb I}}_{{\cal L}_1}
+d_{12}''
\end{equation}
we write
\begin{equation}\label{VIII13-3}
 \E d_{12}{\mathbb I}_{{\cal E}_{12}}
=
\E d'_{12}{\mathbb I}_{{\cal E}_{12}}{\mathbb I}_{{\cal L}_1}+R_1+R_2,
\end{equation}
where
$R_1=\E d''_{12}{\mathbb I}_{{\cal E}_{12}}$
and $R_2= \E {\overline {\mathbb I}}_{{\cal L}_1} d'_{12}{\mathbb I}_{{\cal E}_{12}}$.
Next,
we  observe that
$\E {\mathbb I}_{{\cal L}_1} d'_{12}{\mathbb I}_{{\cal E}_{12}}
=
\E {\mathbb I}_{{\cal L}_1} d'_{1j}{\mathbb I}_{{\cal E}_{1j}}$,
for $2\le j\le n$,
and write
\begin{equation}\label{pb3}
\E {\mathbb I}_{{\cal L}_1} d'_{12}{\mathbb I}_{{\cal E}_{12}}
=
\E {\mathbb I}_{{\cal L}_1} d'_{1*}{\mathbb I}_{{\cal E}_{1*}}
=
\E
{\mathbb I}_{{\cal L}_1} H(m-1)^{-1}.
\end{equation}
We explain the second identity of (\ref{pb3}).
We observe that
 $H(m-1)^{-1}$ is the conditional expectation of
$d'_{1*}{\mathbb I}_{{\cal E}_{1*}}$ given $D_1,\dots, D_n$.
Indeed,
any pair of sets $D_i, D_j$ containing $w_1$ intersects in the
single point $w_1$, since the event ${\cal L}_1$ holds.  Consequently, each $D_i$ containing
$w_1$ produces $X_i-2$ black common neighbours provided
that $w_*$ hits $D_i$. Since the probability that $w_*$ hits $D_i$ equals
 $(X_i-1)/(m-1)$, the  set $D_i$ contributes (on average) $(m-1)^{-1}{\mathbb I}_i(w_1)(X_i-1)_2$
black vertices to $d'_{1*}$.

Now, by the symmetry property, we write the right-hand side of (\ref{pb3}) in the form
\begin{equation}\label{pb4}
\frac{n}{m-1}\E{\mathbb I}_{{\cal L}_1}{\mathbb I}_1(w_1)(X_1-1)_2
=
\frac{n}{m-1}\E{\mathbb I}_1(w_1)(X_1-1)_2-R_3
=
\frac{n}{(m)_2}y_3
-
R_3,
\end{equation}
where, $R_3=\frac{n}{m-1}\E{\overline {\mathbb I}}_{{\cal L}_1}{\mathbb I}_1(w_1)(X_1-1)_2$.
Finally, we observe that (\ref{VIII13}) follows from
(\ref{VIII13-3}), (\ref{pb3}), (\ref{pb4}) and the bounds $R_i=O(n^{-2})$, $i=1,2,3$,
 which are proved below.

In order to bound $R_i$, $i=1,2$,  we use the  inequalities
\begin{equation}\label{VIIIQ1}
 d'_{12}\le d_1\le L_1,
\qquad
{\mathbb I}_{{\cal E}_{12}}\le S_1,
\qquad
{\overline {\mathbb I}}_{{\cal L}_1}
=
{\mathbb I}_{\{L_1\not=d_1\}}={\mathbb I}_{\{Q_1\ge 1\}}
\le
Q_1
\end{equation}
and write $R_2\le \E Q_1L_1S_1$ and $R_3\le n(m-1)^{-1}\E Q_1{\mathbb I}_1(w_1)(X_1-1)_2$.
Then we apply  (\ref{L1Q1S1}) and (\ref{Q1R3}).
In order to bound $R_1$ we
observe, that the number of red common neighbours of $w_1,w_2$ produced by the pair of sets $D_i$, $D_j$
is
\begin{displaymath}
a_{ij}
=
\left(
{\mathbb I}_i(w_1){\mathbb I}_j(w_2){\overline {\mathbb I}}_j(w_1){\overline {\mathbb I}}_i(w_2)
+
{\mathbb I}_j(w_1){\mathbb I}_i(w_2){\overline {\mathbb I}}_i(w_1){\overline {\mathbb I}}_j(w_2)
\right) X_{ij}.
\end{displaymath}
Hence, on the event $w_1,w_2\in D_1$ we have
$d''_{12}\le \sum_{2\le i<j\le n}a_{ij}$,
since elements of $D_1\setminus\{w_1,w_2\}$  are black common neighbours of $w_1,w_2$.
>From this inequality and the inequality ${\mathbb I}_{{\cal E}_{12}}\le S_1$ we obtain
\begin{equation}\label{pb1}
 R_1\le \E d''_{12}S_1
=
n\E d''_{12}s_1\le n\tbinom{n-1}{2}\E s_1a_{23}.
\end{equation}
Furthermore, invoking in (\ref{pb1})  identities
\begin{displaymath}
\E (s_1a_{23})
=
\E {\tilde \E} (s_1 a_{23})
=
\E\left({\tilde \E}s_1\right)\left({\tilde \E}a_{23}\right),
\qquad
{\tilde \E}s_1=(X_1)_2/(m)_2
\end{displaymath}
and
 inequalities
\begin{displaymath}
{\tilde \E}a_{23}
=
2\E
{\mathbb I}_2(w_1){\mathbb I}_3(w_2){\overline {\mathbb I}}_3(w_1){\overline {\mathbb I}}_2(w_2)X_{23}
\le
 2 \frac{X_2}{m}\frac{X_3}{m}\frac{(X_2-1)(X_3-1)}{m-2}
\end{displaymath}
we obtain $R_1=O(n^{-2})$.


{\it Proof of (\ref{passive2+h+k})}.
In the proof we use the fact
that the random vector $(H, L_1)$ converges in distribution to $(d_{2*}, d_{**})$ as $n\to+\infty$.
We recall that $H$ is described after (\ref{pb3}). The proof of this fact is similar to that of
the convergence in distribution
of $L_1=\sum_{1\le i\le n}{\mathbb I}_i(w_1)(X_i-1)$ to the random variable $d_{**}$,
see Theorems 5 and 7 of \cite{Bloznelis2011+}. We note that the convergence in distribution of
$(H,L_1)$ implies the convergence in distribution of  $H{\mathbb I}_{\{L_1=k\}}$ to
$d_{2*}{\mathbb I}_{\{d_{**}=k\}}$. Furthermore, since under condition (v)
the first moment $\E H$ is uniformly bounded as $n\to+\infty$ and $\E d_{2*}<\infty$, we obtain
the convergence of moments
\begin{equation}\label{HL}
 \E H{\mathbb I}_{\{L_1=k\}}\to \E d_{2*}{\mathbb I}_{\{d_{**}=k\}}
\qquad
{\text{as}}
\qquad
n\to\infty.
\end{equation}
In order to prove  (\ref{passive2+h+k}) we write
\begin{displaymath}
 h_k
=
\E(d_{12}|w_1\sim w_2, d_1=k)=p_{ke}^{-1}\E d_{12}{\mathbb I}_{{\cal E}_{12}}{\mathbb I}_{\{d_1=k\}}
\end{displaymath}
and show that
\begin{eqnarray}\label{passive+pke}
&&
 p_{ke}=km^{-1}\PP(d_{**}=k)+o(n^{-1}),
\\
\label{HL++}
&&
\E d_{12}{\mathbb I}_{{\cal E}_{12}}{\mathbb I}_{\{d_1=k\}}
=m^{-1}\E H{\mathbb I}_{\{L_1=k\}}+o(n^{-1}).
\end{eqnarray}
We remark that (\ref{HL}) in combination with  (\ref{passive+pke}) and   (\ref{HL++})
implies (\ref{passive2+h+k}).

Let us show (\ref{passive+pke}).
 In view of the identities
$p_{ke}=\PP(w_i\sim w_1,d_1=k)$, $2\le i\le n$, we can write
\begin{equation}\nonumber
p_{ke}
=
\PP(w_*\sim w_1, d_1=k)
=
\PP(w_*\sim w_1|d_1=k)\PP(d_1=k).
\end{equation}
Now,  from the simple identity
$\PP(w_*\sim w_1|d_1=k)=k(m-1)^{-1}$
and the approximation $\PP(d_1=k)=\PP(d_{**}=k)+o(1)$, see \cite{Bloznelis2011+},
we obtain (\ref{passive+pke}).

Let us show (\ref{HL++}). Using (\ref{identityVIII}) we obtain, cf. (\ref{VIII13-3}),
\begin{equation}\label{HK++1}
\E d_{12}{\mathbb I}_{{\cal E}_{12}}{\mathbb I}_{\{d_1=k\}}
=
\E d'_{12}{\mathbb I}_{{\cal E}_{12}}{\mathbb I}_{\{d_1=k\}}{\mathbb I}_{{\cal L}_1}
+ O(n^{-2}).
\end{equation}
Furthermore, proceeding as in (\ref{pb3}), we obtain
\begin{equation}\label{HK++2}
\E d'_{12}{\mathbb I}_{{\cal E}_{12}}{\mathbb I}_{\{d_1=k\}}{\mathbb I}_{{\cal L}_1}
=
\E d'_{1*}{\mathbb I}_{{\cal E}_{1*}}{\mathbb I}_{\{d_1=k\}}{\mathbb I}_{{\cal L}_1}
=(m-1)^{-1}\E H
{\mathbb I}_{\{d_1=k\}}{\mathbb I}_{{\cal L}_1}.
\end{equation}
Next, we invoke identity $\E H
{\mathbb I}_{\{d_1=k\}}{\mathbb I}_{{\cal L}_1}=\E H
{\mathbb I}_{\{L_1=k\}}{\mathbb I}_{{\cal L}_1}$ and approximate, cf. (\ref{pb4}),
\begin{equation}\label{HK++3}
(m-1)^{-1}\E H
{\mathbb I}_{\{L_1=k\}}{\mathbb I}_{{\cal L}_1}
=
(m-1)^{-1}\E H
{\mathbb I}_{\{L_1=k\}}+O(n^{-2}).
\end{equation}
Combining (\ref{HK++1}), (\ref{HK++2}) and (\ref{HK++3}) we obtain (\ref{HL++}).


{\it Proof of (\ref{passive2+b+k})}.
Let  ${\overline d}_{12}$ denote
the number of neighbours of $w_1$, which are not adjacent to $w_2$, and let
${\overline h}_{k}=\E({\overline d}_{12}|w_1\sim w_2, d_2=k)$.
We obtain (\ref{passive2+b+k}) from the identity
\begin{equation}\nonumber
 b_k=\E(d_1|w_1\sim w_2, d_2=k)=1+h_k+{\overline h}_{k}
\end{equation}
and the relation
${\overline h}_{k}=\beta_n^{-1}y_2+o(1)$.
In order to prove this relation we write
\begin{equation}\nonumber
 {\overline h}_{k}=p_{ke}^{-1}\tau,
\qquad
{\text{where}}
\qquad
\tau=\E {\overline d}_{12}{\mathbb I}_{{\cal E}_{12}}{\mathbb I}_{\{d_2=k\}},
\end{equation}
and combine (\ref{passive+pke}) with the identity
\begin{equation}\label{tau--}
 \tau=km^{-1}\beta_n^{-1}y_2\PP(d_{**}=k)+o(n^{-1}).
\end{equation}
It remains to prove (\ref{tau--}). In the proof we use the shorthand notation
\begin{displaymath}
\eta_i={\mathbb I}_i(w_1){\overline{\mathbb I}}_i(w_2)(X_i-1),
\qquad
\eta_i'=\eta_i{\mathbb I}_{{\cal E}_{12}}{\mathbb I}_{\{d_2=k\}}{\mathbb I}_{{\cal L}_1},
\qquad
\eta_i''= \eta_i{\mathbb I}_{{\cal E}_{12}}{\mathbb I}_{\{d_2=k\}}.
\end{displaymath}
Let us prove (\ref{tau--}). Using the identity
$1={\mathbb I}_{{\cal L}_1}+{\overline{\mathbb I}}_{{\cal L}_1}$
we write
\begin{equation}\nonumber
\tau
=
\E {\overline d}_{12}{\mathbb I}_{{\cal E}_{12}}{\mathbb I}_{\{d_2=k\}}{\mathbb I}_{{\cal L}_1}+R_4,
\qquad
R_4
=
\E{\overline d}_{12}{\mathbb I}_{{\cal E}_{12}}{\mathbb I}_{\{d_2=k\}}
{\overline {\mathbb I}}_{{\cal L}_1}.
\end{equation}
Next, assuming that the event ${\cal L}_1$ holds,
we invoke the identity ${\overline d}_{12}=\sum_{1\le i\le n}\eta_i$ and obtain
\begin{displaymath}
 \E {\overline d}_{12}{\mathbb I}_{{\cal E}_{12}}{\mathbb I}_{\{d_2=k\}}{\mathbb I}_{{\cal L}_1}
=\E\sum_{1\le i\le n}\eta'_i=n\E \eta_1'.
\end{displaymath}
In the last step we used the symmetry property. Furthermore, from the identity
\begin{displaymath}
 \E \eta_1'=\E\eta_1''-R_5,
\qquad
R_5=\E \eta_1{\mathbb I}_{{\cal E}_{12}}{\mathbb I}_{\{d_2=k\}}{\overline{\mathbb I}}_{{\cal L}_1},
\end{displaymath}
we obtain
$\tau=n\E \eta_1''+R_4-nR_5$. We note that inequalities ${\overline d}_{12}\le d_1\le L_1$
and (\ref{VIIIQ1}) imply
\begin{displaymath}
 R_4\le \E L_1S_1Q_1,
\qquad
R_5\le \E {\mathbb I}_1(w_1)(X_1-1)S_1Q_1=n^{-1}\E L_1S_1Q_1.
\end{displaymath}
Now, from (\ref{L1Q1S1}) we obtain $R_4=O(n^{-2})$ and $R_5=O(n^{-3})$.
Hence, we have $\tau=n\E \eta_1''+O(n^{-2})$.
Finally, invoking the relation
\begin{equation}\label{f2}
 \E \eta''_1
=
km^{-2}y_2\PP(d_{**}=k)+o(n^{-2}),
\end{equation}
we obtain (\ref{tau--}). To show (\ref{f2}) we write
\begin{equation}\label{f13}
 \E \eta''_1=\E  \eta_1 \kappa,
\qquad
\kappa=\E\left({\mathbb I}_{{\cal E}_{12}}{\mathbb I}_{\{d_2=k\}}\bigl|D_1\right),
\end{equation}
and observe that on the event ${w_2\notin D_1}$ the quantity $\kappa$ evaluates the probability
of the event $\{w_1\sim w_2, d_2=k\}$ in the passive random intersection graph defined by
the sets $D_2,\dots, D_3$ (i.e., the random graph $G^*_1(n-1,m,P)$). We then apply
(\ref{passive+pke}) to the graph $G^*_1(n-1,m,P)$ and obtain
$\kappa=km^{-1}\PP(d_{**}=k)+o(n^{-1})$. Here the remainder term does not depend on $D_1$.
Substitution of this identity in (\ref{f13}) gives
\begin{equation}\nonumber
 \E \eta''_1=\left(km^{-1}\PP(d_{**}=k)+o(n^{-1})\right)\E \eta_1.
\end{equation}
The following identities complete the proof of (\ref{f2})
\begin{eqnarray}\nonumber
\E \eta_1
&
=
&
\E {\mathbb I}_1(w_1)(X_1-1)-\E {\mathbb I}_1(w_1){\mathbb I}_1(w_2)(X_1-1)
\\
\nonumber
&
=
&
m^{-1}y_2-(m)_2^{-1}(y_3+y_2).
\end{eqnarray}
\end{proof}



{\it Acknowledgement}. The work of M. Bloznelis and V. Kurauskas
was supported in part by the  Research Council of Lithuania grant MIP-053/2011. J.~Jaworski acknowledges the support by 
National Science Centre - DEC-2011/01/B/ST1/03943.

\end{document}